\documentclass[12pt,twoside]{article}
\usepackage{graphicx}
\usepackage[french,english]{babel}
\usepackage[utf8]{inputenc} 
\usepackage[T1]{fontenc} 
\usepackage{array} 
\usepackage{tabularx,booktabs,multirow} 
\usepackage{amsmath, amsfonts, amssymb, amsthm, mathtools, thmtools, mathrsfs} 
\usepackage{ dsfont } 
\usepackage{lmodern} 
\usepackage{ytableau, youngtab, genyoungtabtikz} 
\ytableausetup{aligntableaux=center}
\usepackage{cancel} 
\usepackage{comment}
\usepackage{url}
\usepackage{subfigure}
\usepackage[normalem]{ulem}
\usepackage{makecell}
\usepackage{rotating}
\usepackage{longtable}
\usepackage{fullpage}
\usepackage{pdflscape}
\usepackage[title]{appendix}

\usepackage{imakeidx}
\makeindex[title=Index des définitions, program=makeindex, columns=2]
\let\emphorig\emph
\renewcommand{\emph}[1]{\emphorig{#1}\index{#1}}

\newtheoremstyle{mes_theoremes}{1.5em}{2em}{}{}{\bfseries}{~:~}{\parskip}{\thmname{#1}\thmnumber{ #2}\thmnote{ (#3)}}
\theoremstyle{mes_theoremes} 
\newtheorem{theo}{Theorem}[section]
\newtheorem{theo*}{Theorem}
\newtheorem{prop}[theo]{Proposition}
\newtheorem{conj}[theo]{Conjecture}
\newtheorem{cor}[theo]{Corollary}
\newtheorem{lem}[theo]{Lemma}

\newtheorem{rem}[theo]{Remark}

\newtheorem*{ex*}{Example}
\newtheorem{ex}[theo]{Example}

\usepackage[colorinlistoftodos]{todonotes}
\newcommand{\flo}[1]{\todo[size=\tiny,color=green!30,inline]{#1}}
\newcommand{\etienne}[2][]{\todo[size=\tiny,color=red!30,inline,#1]{#2}}

\newcommand{\N}{\mathbb{N}}

\newcommand{\C}{\mathbb{C}}

\newcommand{\SYT}{\mathrm{SYT}}
\newcommand{\SSYT}{\mathrm{SSYT}}

\Yboxdim{0.35cm} 
\newcommand\rd{\Yfillcolour{red}}
\newcommand\bl{\Yfillcolour{blue}}

\newcommand\wt{\Yfillcolour{white}}

\usepackage{tikz, tikz-3dplot, pgfplots}
\usepackage{tkz-graph}
\usepackage{ulem}
\usetikzlibrary[positioning,patterns,shapes, calc]

\pgfdeclareshape{3rRevL}
{
\savedanchor\centerpoint{%
\pgf@x=.5\wd\pgfnodeparttextbox%
\pgf@y=.5\ht\pgfnodeparttextbox%
\advance\pgf@y by -.5\dp\pgfnodeparttextbox%
}
\anchor{center}{\centerpoint}
\anchorborder{\centerpoint}
\anchor{text}{%
\centerpoint%
\pgf@x=-\pgf@x%
\pgf@y=-\pgf@y%
}
}

\pgfplotsset{compat=1.8}
\tikzstyle{bsq}=[rectangle, draw, thick, minimum width=1cm, minimum height=1cm] 
\tikzstyle{3rver}=[rectangle, draw, thick, minimum width=1cm, minimum height=3cm]
\tikzstyle{3rhor}=[rectangle, draw, thick, minimum width=3cm, minimum height=1cm]
\def\3rRevL at (#1,#2){\draw (#1,#2) -- ++(0,2) -- ++(2,0) -- ++(0,-1) -- ++(-1,0) -- ++(0,-1) -- cycle;}
\def\3rJ at (#1,#2){\draw (#1,#2) -- ++(0,1) -- ++(1,0) -- ++(0,1) -- ++(1,0) -- ++(0,-2) -- cycle;}

\begin{document}

\title{
3-Plethysms of homogeneous and elementary symmetric functions}

\author{Florence Maas-Gariépy, Étienne Tétreault }

\thispagestyle{empty}        
\maketitle


\begin{abstract}
We introduce the new combinatorial approach of plethystic type of tableaux, as a method to understand coefficients of Schur functions appearing in plethysms $s_\nu[h_\lambda]$ and $s_{\nu}[e_{\lambda}]$, for any partitions $\lambda$ and $\nu$. We first give general results about this approach, then use results on tableaux, ribbon tableaux and integer points in polytopes to understand the case where $\nu$ is a partition of $3$ and $\lambda$ has one part. We then use a \textit{Kronecker map} to extend these results to any partition $\lambda$.
\end{abstract}



Plethysm of symmetric functions is an operation which arises naturally in representation theory as the character of a composition of representations. In general, we are interested in understanding plethysms $f[g]$ for any two symmetric functions $f,g$, which are characters of polynomial representations of the general linear group. The plethysm $f[g]$ is also a character, so we are interested in understanding its decomposition into the basis of Schur functions $s_\lambda$, as these are the irreducible characters of the general linear group, which are indexed by partitions $\lambda$. The general problem can be reduced (slightly) to that of understanding $s_\nu[g]$ for any Schur function $s_\nu$, but this remains extremely difficult, and an open problem. 
Our approach exploits the classical decomposition into irreducibles of the $m^{\rm th}$-tensor power representation. Under composition, this decomposition translates to the decomposition of the $m^{\rm th}$-tensor power of any representation. Coined as plethysm of characters, this gives the symmetric function identity
$g^m = \sum_{\nu\vdash m} f^\nu s_\nu[g]$, where $f^\nu$ is the number of standard tableaux of shape $\nu$. Therefore the study of powers of symmetric functions can enlighten us in our quest to understand plethysm.

In this article, we completely describe the plethystic decomposition of $g^m$ when $m=3$ and $g$ is either a homogeneous symmetric function $h_{\lambda}$ or an elementary symmetric function $e_{\lambda}$ . We also give partial advances on the general question, for any $m$. 
After a review of basic notions about symmetric functions in section~\ref{section:background}, we explain our combinatorial approach of plethystic type in section~\ref{section:plethysm}, and state general results in section~\ref{section:general}. We then use these rules in the case $h_n^3$ in section~\ref{sec:tableaux}. We establish results on ribbon tableaux and integer points in polytopes in sections~\ref{section:ribTab} and \ref{section:polytopes}, and use them to understand the decomposition of $h_n^3$ in section~\ref{section:Types}. There are already formulas for the coefficients of these plethysms, known as Chen's formulas \cite{Chen}, but we give an explicit combinatorial description. By using the $\omega$ involution on symmetric functions, we can also understand the decomposition of $e_n^3$ in section~\ref{section:elementary}. Finally, using properties of plethysm, Kronecker coefficients and jeu de taquin, we use these results to understand the decomposition of $h_{\lambda}^3$ and $e_{\lambda}^3$ in section~\ref{section:Product}. This last part extends the results of 
\cite{MaasTetreault}, which introduced the description of the plethystic decomposition of $h_{\lambda}^2$ and $e_{\lambda}^2$ in terms of plethystic type of tableaux.

\section{Background}
\label{section:background}

As much as possible, results, definitions and notations will be introduced when needed. The following gives a very basic overview of the notations and the definitions that are used throughout the article. We refer to \cite{Fulton}, \cite{Sagan} or \cite{Stanley} for more detailed descriptions.\\

Recall that partitions are weakly decreasing sequences of integers. If all parts of a partition $\lambda = (\lambda_1, \lambda_2, \ldots, \lambda_k)$ add up to $n$, we say that $\lambda$ is a partition of $n$, denoted $\lambda\vdash n$. We identify partitions with their diagrams, the left- and top-justified arrays of boxes with $\lambda_i$ cells in the $i^{th}$ row. We can also form skew partitions $\lambda/\mu$ by removing the cells of $\mu$ from $\lambda$, if $\mu$ is contained in $\lambda$. A filling of the cells of a diagram $\lambda$ by positive integers is called a \textit{tableau} of shape $\lambda$. We say that a tableau is \textit{semistandard} if its rows weakly increase from left to right, and its columns strictly increase from top to bottom. Unless otherwise stated, we use the word tableau to mean semistandard tableau further on, and we denote by $\SSYT(\lambda)$ the set of tableaux of shape $\lambda$. If every entry from $1$ to $n$ appears exactly once, we say that the tableau is \textit{standard}, and we denote it by $t$ or other lowercase letters to distinguish them. \\

The \textit{content} of a tableau $T$ is the sequence $(\beta_1, \beta_2, \ldots)$, where $\beta_i$ is the number of entries $i$ in $T$. For $x = (x_1, x_2, \ldots)$, the monomial associated to $T$ is $x^T = x_1^{\beta_1} x_2^{\beta_2} \ldots $. 


\begin{figure}[h]
    \centering
    $T =  \young(11123,22334,346,46) \in \SSYT(5,5,3,2)$, and $x^T=x_1^3x_2^3x_3^4x_4^3x_6^2$. \\
    
    
    \caption{Tableau $T$ 
    of shape $(5,5,3,2)$,
    content $(3,3,4,3,0,2)$ and associated monomial $x^T$.}
    \label{fig:my_label}
\end{figure}

\textit{Schur functions} $s_\lambda$ are defined as the sum of the weights of all tableaux of shape $\lambda$: $s_\lambda = \sum_{T\in \SSYT(\lambda) } x^T$. \textit{Homogeneous symmetric functions} $h_{\lambda}$ are defined as the product $h_{\lambda_1}h_{\lambda_2}\ldots h_{\lambda_k}$, where $h_n = s_{(n)}$. \textit{Elementary symmetric functions} $e_{\lambda}$ are defined as the product $e_{\lambda_1}e_{\lambda_2}\ldots e_{\lambda_k}$, where $e_n = s_{(1)^n}$. 
 The sets $\{s_\lambda \ | \ \lambda \text{ partition} \}$, $\{h_\lambda \ | \ \lambda \text{ partition}\}$ and $\{e_{\lambda} \ | \ \lambda \text{ partition}\}$ are three bases of the \textit{algebra of symmetric functions}: formal sums on a countable set of variables $x=(x_1,x_2, \ldots)$, with coefficients in $\C$, such that exchanging any two variables gives back the original formal sum. 
 In particular, there is a scalar product $\langle , \rangle$ on this algebra, for which Schur functions are orthonormal.\\

We use the Pieri rule to describe the power of homogeneous symmetric functions:
\begin{prop}
For any partition $\mu$ and positive integer $n$:
\[
s_{\mu}h_n = \displaystyle \sum_{\nu} s_{\nu},
\]
where the sum is over all partitions $\nu$ such that $\nu / \mu$ has $n$ cells and does not have two cells in the same column.
\end{prop}

Using this rule, we can easily derive the following formula:

\[
h_n^m = \displaystyle \sum_{\mu \vdash mn} K_{(n)^m}^{\mu} s_{\mu},
\]

where $K_{(n)^m}^{\mu}$ are \textit{Kostka numbers}, counting the number of tableaux of shape $\mu$ and content $(n)^m=(n,n,\ldots,n)$, each entry $1$ to $m$ appearing $n$ times. The set of such tableaux is denoted $\SSYT(\mu,(n)^m)$.

\section{Plethysm and plethystic type}\label{section:plethysm}

\textit{Plethysm} is a binary operation on symmetric functions, denoted $f[g]$. The easiest way to define it is via \textit{power sum symmetric functions} $p_k = \sum_{i=0}^{\infty} x_i^k$, which form an algebraic basis of the ring of symmetric functions, just as the $h_n$'s and $e_n$'s respectively do.
Plethysm is then the unique operation such that, for any symmetric functions $f,g,h$ and any nonnegative integers $k,m$:

\begin{itemize}
    \item $p_k[p_m] = p_{km}$,
    \item $p_k[f + g] = p_k[f] + p_k[g]$,
    \item $p_k[f\cdot g] = p_k[f] \cdot p_k[g]$,
    \item $(f + g)[h] = f[h] + g[h]$,
    \item $(f\cdot g)[h] = f[h] \cdot g[h]$.
\end{itemize}

Using these rules, one can show the following equality, which is a standard result; we call it the \textit{plethystic decomposition} of $g^m$.

\begin{prop}\label{Plethystic types}
For every symmetric function $g$, and $m\in \N$,
\[g^m = (p_1[g])^m = \displaystyle \sum_{\nu \vdash m} f^{\nu} s_{\nu}[g],\]

\noindent where $f^{\nu} = K_{(1)^m}^{\nu}$ is the number of standard tableaux of shape $\nu$.
\end{prop}

Denote $\SYT_m$ the set of standard tableaux with $m$ cells. We have the equality:

\[
h_n^m = \displaystyle \sum_{\mu \vdash mn} K_{(n)^m}^{\mu} s_{\mu} = \displaystyle \sum_{t \in \SYT_m} s_{\text{shape}(t)}[h_n]
\]

This equality means that for any $\mu \vdash mn$, there is a way to partition the set $\SSYT(\mu,(n)^m)$ into distinct subsets $\SSYT_t(\mu,(n)^m)$, one for each standard tableau $t \in \SYT_m$, such that $\langle s_{\mu}, s_{\text{shape}(t)}[h_n] \rangle = \left| \SSYT_t(\mu,(n)^m) \right|$. We call such a partition a \textit{type attribution}.\\

We say that a tableau $T$ has \textit{type} $t$ when $T \in \SSYT_t(\mu,(n)^m)$, and that the copy of $s_{\mu}$ indexed by $T$ (in $h_n^m$) contributes to the copy of $s_{\nu}[h_n]$ indexed by 
$t$, if $t$ has shape $\nu$.

Attributing a type to every tableau of shape $(n)^m$ can then give a combinatorial description of the plethystic decomposition of $h_n^m$. However, how to do so explicitly is a difficult problem, and remains open in general. \\

This article is about the case $m=3$, so we want to attribute to tableaux of content $(n)^3$ either the type $\young(123)$, $\young(12,3)$, $\young(13,2)$ or $\young(1,2,3)$. 
We explore conditions that must be respected by this attribution, and describe the simplest, and most elegant, type attribution for tableaux of filling $(n)^3$. We then use this result to understand the plethystic decomposition of $e_n^3$, $h_{\lambda}^3$ and $e_{\lambda^3}$. \\

In order to navigate between plethysms on homogeneous and elementary symmetric functions, we use the involution $\omega$ on symmetric functions defined on Schur functions by $\omega(s_\lambda) = s_{\lambda'}$, where $\lambda'$ is the \textit{conjugate} of $\lambda$: the partiton obtained from $\lambda$ by reflecting it along its main diagonal, exchanging lengths of rows and columns. As $h_n=s_{(n)}$ and $e_n=s_{(1)^n}$, we have that $\omega(h_n)=e_n$. Note that this involution preserves the scalar product. It is possible to show (\cite{Macdonald}, chapter 1, example 8.1) that:
\begin{center}
	$\omega(s_\nu[s_\gamma]) = \bigg\{ $ \begin{tabular}{cl} $s_\nu[s_{\gamma'}]$ & \text{ if } $\gamma \vdash n$ \text{ for } $n$ \text{ even } \\  $s_{\nu'}[s_{\gamma'}]$ & \text{ if } $\gamma \vdash n$ \text{ for } $n$ \text{ odd } \end{tabular} , 
\end{center}

In section \ref{section:elementary}, we show how the $\omega$ involution translates a plethystic decomposition of $h_n^m$ into a plethystic decomposition for $e_n^m$.


\section{General rules for attributing types}\label{section:general}

Before we study the case $m=3$, we study the general case. This allows us to have rules that can be applied in our case, but also guide future researches when $m > 3$.

We show that we can use a recursive process to describe type attribution for $m+1$ in terms of that for $m$, and that we can reduce our study by decomposing tableaux into simpler ones, while describing the corresponding change of plethystic type.

\subsection{Types and $\ell$-subtypes}

Using the formulas above, we have that

\vspace{-1em}
\begin{align*}
h_n^{m+1} &= (h_n)^{m}\cdot h_n \\
&= \left( \displaystyle \sum_{\nu\vdash n\cdot m} K^\nu_{(n)^m}s_{\nu} \right) \cdot h_n \\
&= \left( \displaystyle \sum_{t \in \SYT_{m}} s_{\text{shape}(t)}[h_n] \right) \cdot h_n \\
&= \displaystyle \sum_{t \in \SYT_{m}} (s_{\text{shape}(t)} \cdot h_1)[h_n],
\end{align*}
where the last line is obtained by using $h_n=s_n=s_1[h_n]$, and the product rule of plethysm.\\

For any tableau $T$, denote $T\downarrow_{\ell}$ its subtableau with entries $1$ to $\ell$. 
%
Suppose $\bar{T} \in \SSYT(\mu,(n)^{m+1})$ is such that $\bar{T}\downarrow_m = T\in \SSYT(\nu,(n)^m)$. Then the Schur function associated to $\bar{T}$ in the Schur expansion of $h_n^{m+1} = h_n^m \cdot h_n$ comes from the product $s_{\nu} \cdot h_n$, and $\bar{T}$ is obtained from $T$ by adding a horizontal band of $n$ entries $m+1$.\\ 

Now, suppose the type of $\bar{T}$ is $\bar{t}\in \SYT_{m+1}$, such that $\bar{t}\downarrow_m = t\in SYT_{m}$. Then the type of $T$ is $t$, as multiplying $s_{\text{shape}(t)}$ by $h_1$ corresponds to adding a cell to $t$, and so the type of $\bar{T}$ is obtained from the type of $T$ by adding one cell.\\ 

Using the above result recursively, we obtain an important rule for attributing types: \textit{if $T$ has type $t$, then the type of $T\downarrow_{\ell}$ is $t\downarrow_{\ell}$ for any $\ell$}. For this reason, we call $t\downarrow_{\ell}$ the $\ell$\textit{-subtype} of $T$. 
Using this rule, the plethystic decomposition of $h_n^{m+1}$ can be understood in terms of that of $h_n^m$, by understanding how adding $n$ cells to a tableau adds a cell to its type.\\

Let's consider previously known results for $m<3$. When $m=1$, the $1$-subtype is always \young(1). 
When $m=2$, we have 
the following well-known formulas, attributed to Littlewood \cite{Littlewood}, which describe the $2$-subtypes (for the proof, see \cite{Macdonald}, chapter 1, example 8.9):

\[s_{2}[h_n]= \sum_{0\leq j\leq \lfloor \frac{n}{2} \rfloor} s_{2n-2j, 2j}, \] 
\[s_{11}[h_n]= \sum_{0\leq j\leq \lfloor \frac{n}{2} \rfloor} s_{2n-(2j+1), (2j+1)}. \] 

This means that if $T\in \SSYT(\nu,(n)^m$ is such that $T\downarrow_2\in \SSYT(\mu,(n)^2$, for $\mu = (\mu_1,\mu_2) \vdash 2n$, the $2$-subtype of $T$ is $\young(12)$ if $\mu_2$ is even, and $\young(1,2)$ if $\mu_2$ is odd. \\ 

Now, we describe how we can decompose tableaux of filling $(n)^m$ into simpler tableaux, while keeping track of changes of plethystic type.
To do so, we need the following notation. \\

If $T_1$ and $T_2$ are two tableaux of respective shape $\mu^{(1)}$ and $\mu^{(2)}$, denote $T_1 \vee T_2$ the tableau of shape $\mu^{(1)} + \mu^{(2)}$ (where addition is component-wise) obtained by concatenating each row, and reorder the entries in the rows so that they appear in weakly increasing order (if the result is not a tableau, then $T_1 \vee T_2 = \varnothing$). For example:

\[
\young(112,23,3) \ \vee \ \young(1123,23) = \young(1111223,2233,3).
\]
We also denote $T^{\vee k}$ for $\underbrace{T \vee \ldots \vee T}_{k \text{ times}}$. 

\subsection{Adding columns of height $m$}

Let $\mu = (\mu_1, \ldots, \mu_m)\vdash n\cdot m$, 
and $T \in \SSYT(\mu,(n)^m)$. Then $T$ is of the form $\young(1,\vdots,m)^{\vee \mu_m} \vee \widehat{T}$, where $\widehat{T} \in \SSYT(\mu - (\mu_m)^m,(n-\mu_m)^m)$ has at most $m-1$ rows.


The following proposition states that the type of $T$ is then determined by the type of $\widehat{T}$ and the number of columns. The proof uses results of \cite{deBoeckPagetWildon}, and solves a previous conjecture in Chapter 4 of de Boeck's thesis \cite{DeBoeck}.

\begin{prop}

For every $\mu \vdash mn$ and $\nu \vdash m$, we have 

\[\langle s_{\nu'}[h_{n+1}], s_{(1)^m+\mu} \rangle = \langle s_\nu[h_n], s_{\mu} \rangle.\]



\label{prop:columns}
\end{prop}



\begin{proof}
	
We use a special case of theorem 1.1 of \cite{deBoeckPagetWildon}. 
It gives us that \[\langle s_\nu[s_{(1^{n+1})}], s_{(m)\cup\mu} \rangle = \langle s_\nu[s_{(1)^n}], s_{\mu} \rangle,\]
where $(m)\cup\mu$ is the partition obtained by adding a part of length $m$ to $\mu$.  

If we suppose that $n$ is even, the $\omega$ involution defined in section~\ref{section:plethysm} gives that

\begin{center}
\begin{tabular}{ccc}
	$\langle s_\nu[s_{(1)^{n+1}}], s_{(m)\cup\mu} \rangle $ & $=$ & $ \langle s_\nu[s_{(1)^n}], s_{\mu} \rangle$\\
	$\omega\downarrow$&&$\downarrow\omega$\\
	$\langle s_{\nu'}[s_{n+1}], s_{(1)^m+\mu'} \rangle $ & $=$ & $ \langle s_\nu[s_{n}], s_{\mu'} \rangle$\\
\end{tabular}.
\end{center}


If $n$ is odd, then $\nu$ and $\nu'$ in the bottom equation are interchanged. The desired result therefore arises no matter the parity of $n$.
\end{proof}

This proposition gives us a rule to attribute types : \textit{the type of $\young(1,\vdots,m) \vee T$ is the conjugate of the type of $T$}. Thus, the type of $T=\young(1,\vdots,m)^{\vee \mu_m}\vee \widehat{T}$ is the same as the one of $\widehat{T}$ if $\mu_m$ is even, and its conjugate 
if $\mu_m$ is odd. 

\subsection{Adding elements in the first row}

We also consider the following result, which is a special case of a theorem from Brion \cite{Brion}:

\begin{prop}\label{prop:FirstRow}
Let $\nu$ be a partition of an integer $m$, and $\mu \vdash mn$. Then:
\[
\langle s_{\nu}[h_n], s_{\mu} \rangle \leq \langle s_{\nu}[h_{n+1}], s_{\mu+(m)} \rangle,
\]
with the equality when $n$ is large enough.
\end{prop}

Note that $T \mapsto T \vee \young(1\ldots m)$ is an injection from $\SSYT(\mu,(n)^m)$ to $\SSYT(\mu+(m),(n+1)^m)$. 
This proposition then gives us another rule for types: \textit{the type of $T \vee \young(1\ldots m)$ is the same as the type of $T$}.
More generally, let $T \in \SSYT(\mu,(n)^m)$, and suppose that there exists a positive integer $k$ and a tableau $\widetilde{T} \in \SSYT(\mu-(km),(n-k)^m)$ such that $T = \widetilde{T} \vee \young(1\ldots m)^{\vee k}$. Then, the type of $T$ is the same as the type of $\widetilde{T}$.\\ 

Let's recall the three rules found so far to attribute types to a tableau $T$ of content $(n)^m$:
\begin{itemize}
\item The $\ell$-subtype of $T$ is the type of $T\downarrow_{\ell}$.
\item The type of $\young(1,\vdots,m) \vee T$ is the conjugate of the type of $T$;
\item The type of $T \vee \young(1\ldots m)$ is the same as the type of $T$;
\end{itemize}

We use them in the next section for $m=3$.

\section{Tableaux of content $(n^3)$}\label{sec:tableaux}

Our goal in this section is to understand the characteristics of the tableaux $T \in \SSYT(\mu,(n)^3)$, for any $\mu \vdash 3n$, and explore what should be considered when attributing a type to such tableaux.

\subsection{Number of tableaux}

There is a very simple formula for the coefficients $K_{(n)^3}^\mu$ of the Schur functions in $h_n^3$, which can be found in \cite{Thrall}:
\begin{prop}
The Kostka number $K_{(n)^3}^\mu$ counting tableaux of shape $\mu$ and content $(n)^3  = (n,n,n)$ is given by
\[K_{(n)^3}^\mu = \min(\mu_1-\mu_2, \mu_2-\mu_3)+1\]
when $\mu$ has at most three parts, and $0$ otherwise.
\label{prop:spanningAllTabOfACertainShape}
\end{prop}

\begin{proof}

A tableau $T$ of content $(n^3)$ must have at most three rows, so let $\mu = (\mu_1,\mu_2,\mu_3) \vdash 3n$ be its shape. Recall that columns of height 3 must be of the form $\young(1,2,3)$, and all $1$'s must be on the first row. There are then two cases to consider. 

Suppose that $\mu_2 - \mu_3 \leq \mu_1 - \mu_2$, which is equivalent to $\mu_2 \leq n$. Then, we can fill the second row entirely with $2$'s. All other tableaux can be obtained by exchanging some of the 2's 
with 3's on the first row. So, there are $\mu_2-\mu_3+1$ possible tableaux.

Otherwise, $\mu_2 - \mu_3 > \mu_1 - \mu_2$, which is equivalent to $\mu_2 > n$. Then, we can fill the first row entirely with 1's and 2's.
We obtain all other tableaux by exchanging some 2's in the first row with 3's on the second, but at least $\mu_2 - n$ 2's must remain since they cover 3's. So, there are $\mu_1 - n - (\mu_2 -n) + 1 = \mu_1 - \mu_2 + 1$ possible tableaux. 

In both cases, the number of tableaux is exactly $\text{min}(\mu_1-\mu_2,\mu_2-\mu_3) + 1$.
\end{proof}

For example, let $\mu = (8,6,1) \vdash 15$. In the illustration below, there is $\mu_2 - \mu_1 = 8 - 6 = 2$ cells in red, and $\mu_3 - \mu_2 = 6 - 1 = 5$ cells in blue. We can always fill the smallest of the two parts with 2's; in this example, it corresponds to the red part. Then, we can exchange each of these 2's, starting with the rightmost one, with 3's that are in the other part (here the blue part). So, there are $\mu_2 - \mu_1 + 1 = 3$ tableaux of shape $\mu = (8,6,1)$.

\begin{center}
\gyoung(;1;1;1;1;1;2!\rd;2;2,!\wt;2!\bl;2;3;3;3;3,!\wt;3) $\rightarrow$ \gyoung(;1;1;1;1;1;2!\rd;2;3,!\wt;2!\bl;2;2;3;3;3,!\wt;3) $\rightarrow$ \gyoung(;1;1;1;1;1;2!\rd;3;3,!\wt;2!\bl;2;2;2;3;3,!\wt;3)
\end{center}

Therefore, we can consider the (unique) tableau which maximizes the number of $2$'s on the first 
row, and obtain all the others by making exchanges with $3$'s on the 
second row.

\subsection{Chen's formulas}

The formulas below are due to Chen \cite{Chen}, 
and give the number of copies of $s_\mu$ that lie in each part of the plethystic decomposition of $h_n^3$. The formula for $s_3[h_n]$ was already known to Thrall \cite{Thrall}, although given in another form. These formulas say that roughly a sixth of the tableaux in $\SSYT(\mu,(n)^3)$ must be of type $\young(123)$, a sixth of type $\young(1,2,3)$, and a third of each type $\young(12,3)$ and $\young(13,2)$. \\

For a fixed $\mu\vdash 3n$, we have 
$K^\mu_{(n)^3} = \min(\mu_1-\mu_2, \mu_2-\mu_3) +1 $. Chen's formulas \cite{Chen} state that the coefficients appearing in the plethysms of $h_n^3$ can be computed as following.\\

For $s_3[h_n] = \displaystyle \sum_{\substack{\mu \vdash 3n \\ \ell(\mu) \leq 3}} c_\mu s_\mu$: $\ \ \ c_\mu = \left\{
\begin{tabular}{cl}
      \large $\left\lceil \frac{
      K^\mu_{(n)^3}}{6} \right\rceil$ & \normalsize \multirow{2}{8cm}{if $(
      K^\mu_{(n)^3} \mod 6) \in \{0,4\}$ OR \newline \quad $(
      K^\mu_{(n)^3} \mod 6) \in \{1,3,5\}$ AND $\mu_2$ even\ } \\
      &\\
      &\\
      \large $\left\lfloor \frac{
      K^\mu_{(n)^3}}{6} \right\rfloor$ & \normalsize if neither above condition hold.\\
\end{tabular} \right.,$

\vspace{1em}
For $s_{111}[h_n] = \displaystyle \sum_{\substack{\mu \vdash 3n \\ \ell(\mu) \leq 3}} c_\mu s_\mu$: $c_\mu = \left\{
\begin{tabular}{cl}
      {\large $\left\lceil \frac{
      K^\mu_{(n)^3}}{6} \right\rceil$} & \normalsize \multirow{2}{8cm}{if $(
      K^\mu_{(n)^3} \mod 6) \in \{0,4\}$ OR \newline \quad $(
      K^\mu_{(n)^3} \mod 6) \in \{1,3,5\}$ AND $\mu_2$ odd\ } \\
      &\\
      &\\
      \large $\left\lfloor \frac{
      K^\mu_{(n)^3}}{6} \right\rfloor$ & \normalsize if neither above condition hold.\\
\end{tabular} \right.,$

\vspace{1em}
For $s_{21}[h_n] = \displaystyle \sum_{\substack{\mu \vdash 3n \\ \ell(\mu) \leq 3}} c_\mu s_\mu$: $\ c_\mu = $ {\large $\left\lceil \frac{
K^\mu_{(n)^3}-1}{3} \right\rceil$}.\\

These formulas were proven using the SXP algorithm \cite{Chen}. Thus, we have a way to prove if a choice of type for all tableaux of content $(n)^3$ is valid. Note that in Chen's article, the notation $\delta_\nu$ is used for $K^\nu_{(n)^3}$.

\subsection{Using general results}

We have seen that 
the $2$-subtype of a tableau $T$ of content $(n)^3$ is given by the parity of the number of $2$'s in its second row. Then,
\begin{itemize}
\item $T$ has type $\young(123)$ or $\young(12,3)$ if the number of 2's in its second row is even;
\item $T$ has type $\young(13,2)$ or $\young(1,2,3)$ if the number of 2's in its second row is odd.
\end{itemize}

About a half of tableaux will have either of these $2$-subtypes. 
Among tableaux with $2$-subtype $\young(12)$ (resp. $\young(1,2)$), about a third should be of type $\young(123)$ (resp. $\young(1,2,3)$), and about two thirds, type $\young(12,3)$ (resp. $\young(13,2)$), to respect Chen's rule.\\

We use rules of section~\ref{section:general} to restrict further our study. 
We have seen that any tableau $T \in \SSYT(\mu,(n)^3)$, for $\mu = (\mu_1,\mu_2,\mu_3)$, can be expressed as $\young(1,2,3)^{\vee \mu_3} \vee \widehat{T}$. 
Then,
\begin{itemize}
\item If $\mu_3$ is even, the type of $T$ is the same as the type of $\widehat{T}$;
\item If $\mu_3$ is odd, the type of $T$ is the conjugate of the type of $\widehat{T}$.
\end{itemize}


We can moreover break down $\widehat{T}$ into $\widehat{T} = \bar{T}\vee \young(123)^{\vee k}$ for a certain $k$. 
Then $\widehat{T}$ and $\bar{T}$ have the same type. When $k$ is maximal, then the minimal tableaux $\bar{T}$ have interesting properties, which we study in section~\ref{section:polytopes} (see Figure~\ref{fig:ReductionTab3cellules1eLigne}).\\

We now describe the construction of a polytope which integer points represent tableaux of content $(n)^3$. The coordinates of these integer points are given by the characteristics of what we call Yamanouchi $3$-ribbon tableaux.

\section{3-ribbon tableaux}
\label{section:ribTab}
In this section, we define Yamanouchi $3$-ribbon tableaux, and show they are in bijection with tableaux of content $(n)^3$.

\subsection{General definition of $r$-ribbon tableaux}

A \textit{$r$-ribbon} is a connected skew diagram with $r$ cells, and no squares (sets of $2\times2$ boxes). Its \textit{head} is its northeast-most cell. 
The $10$-ribbon below has its head marked in red.
\[ \gyoung(:::::;;;!\rd;,!\wt::::;;,:;;;;!\wt)
\]

We say that a shape $\lambda/\mu$ is \textit{pavable by $r$-ribbons} if there is a sequence of skew shapes $\mu = \lambda_0 \subseteq \lambda_1 \subseteq \lambda_2 \subseteq \ldots \subseteq \lambda_k = \lambda$ such that each $\lambda_{i-1}/\lambda_{i}$ is a $r$-ribbon. Each ribbon in the paving 
can be filled by an positive integer to form a $r$-ribbon tableau. Its content is the composition $\alpha = (\alpha_1, \alpha_2, \ldots)$ where $\alpha_i$ is the number of ribbons with entry $i$. 

We say that a $r$-ribbon tableau is semistandard if the subset of 
ribbons with entry $i$ form a \textit{horizontal band}: a sequence of partitions $\lambda_0 \subseteq \lambda_1 \subseteq \lambda_2 \subseteq \ldots \subseteq \lambda_k$, such that $\lambda_{j}/\lambda_{j-1}$ is a ribbon $\xi_j$ with entry $i$, and the head of $\xi_j$ lies weakly northeast of that of $\xi_{j-1}$.\\ 

We can define the reading word of a ribbon tableau to be the word read off by reading rows left to right, bottom to top, recording a ribbon only when its head is scanned. A reading word is said to be \textit{Yamanouchi} (or \textit{reverse lattice}) if the content of each of its suffix is a partition. A ribbon tableau with Yamanouchi reading word is said to be Yamanouchi.
Using this reading order, the semistandard $3$-ribbon tableau below has reading word $3333231111$, which is not Yamanouchi, since it admits a suffix whose content is not a partition ($(4,0,1)$).

\vspace{0em}
\begin{center}
	\resizebox{4.5cm}{!}{
		\begin{tikzpicture}[node distance=0 cm,outer sep = 0pt]
		\tikzstyle{every node}=[font=\LARGE]
		\node[3rver]  (1) at (1.5,3)     {1};
		\node[3rver]  (2) [below = of 1] {3};
		\draw[thick] (2,2.5) -- ++(0,2) -- ++(2,0) -- ++(0,-1) -- ++(-1,0) -- ++(0,-1) -- cycle; 
		\node         (3) at (2.5,4)     {1};
		\node[3rhor]  (4) at ( 5.5, 4)   {1};
		\node[3rver]  (5) at ( 2.5, 1)   {3};
		\node[3rhor]  (6) at (8.5,4)     {1};
		\node[3rhor]  (7) at ( 4.5, 3)   {2};
		
		\draw[thick] (3,0.5) -- ++(0,2) -- ++(2,0) -- ++(0,-1) -- ++(-1,0) -- ++(0,-1) -- cycle; 
		\node        (8)  at (3.5,2)     {3};
		
		\draw[thick] (4,0.5) -- ++(0,1) -- ++(1,0) -- ++(0,1) -- ++(1,0) -- ++(0,-2) -- cycle; 
		\node        (9)  at (5.5,1)     {3};
		
		\draw[thick] (6,1.5) -- ++(0,2) -- ++(2,0) -- ++(0,-1) -- ++(-1,0) -- ++(0,-1) -- cycle; 
		\node        (10)  at (6.5,3)     {3};
		
		\end{tikzpicture}
	}
\end{center}

\subsection{Yamanouchi $3$-ribbon tableaux}

Carré and Leclerc described in \cite{CarreLeclerc} the product of two Schur functions in terms of domino tableaux ($2$-ribbon tableaux).
They showed that the number of Yamanouchi domino tableaux of a certain shape $I$ and content $\lambda$ give the multiplicity $c_{\mu \nu}^\lambda$ of $s_\lambda$ in $s_\mu s_\nu$. 

For a square $s_\mu^2$, they 
showed that the parity of the cospin of the Yamanouchi domino tableaux determines whether the associated Schur function $s_\lambda$ contributes to $s_2[s_\mu]$ 
or $s_{11}[s_\mu]$, where the cospin equal the number of horizontal dominoes divided by two.\\ 

We generalize their strategy to $3$-ribbon tableaux. The bijection described by Carré and Leclerc is a particular case of a bijection given by Stanton and White \cite{StantonWhite}, between $r$-tuples of tableaux and $r$-ribbon tableaux. The $3$-ribbon version gives a bijection between triples of tableaux and $3$-ribbon tableaux. 
In particular, when all tableaux of the triple have shape $(n)$, the corresponding $3$-ribbon tableaux have shape $(3n)^3 = (3n,3n,3n)$.
This approach is interesting as it has the potential to generalize to $r$-ribbon tableaux (and so $r$-plethysms).\\ 

We define Yamanouchi $3$-ribbon tableaux of shape $(3n)^3$ to be built out of blocks of types illustrated below, with the following conditions: if $\kappa_i$ is the number of blocks of type $i$, then 
\begin{itemize}
	\item $0\leq \kappa_i\leq n$ with $0\leq\kappa_3\leq 1$, 
	\item $\kappa_4\leq \kappa_1+\kappa_2$ and 
	\item $\sum_i \kappa_i =n$.
\end{itemize}

\vspace{1em}
\hspace{-3em}
\resizebox{!}{3.25cm}{
\begin{tikzpicture}[node distance=0 cm,outer sep = 0pt]

        \tikzstyle{every node}=[font=\LARGE]
        \node[3rver]  (1) at (0,0)       {1};
        \node[3rver]  (2) [right = of 1] {1};
        \node[3rver]  (3) [right = of 2] {1};
        \node[rectangle, draw, thick, minimum height= 3cm, minimum width = 1cm] (4) [right = of 3]                          {$\hdots$};
        \node[3rver]  (5) [right = of 4] {1};
        \node[3rver]  (6) [right = of 5] {1};
        \node[3rver]  (7) [right = of 6] {1};
        
        \node[3rver]  (8) [right = of 7] {1};
        
        \draw[thick]         (7.6,-0.5) -- ++(0,2) -- ++(2,0) -- ++(0,-1) -- ++(-1,0) -- ++(0,-1) -- cycle; 
        \node         (9) at (8,1)       {1};
        \draw[thick]         (8.6,-.5) -- ++(0,1) -- ++(1,0) -- ++(0,1) -- ++(1,0) -- ++(0,-2) -- cycle; 
        \node        (10)  at (10,0)      {1};
        \node[3rhor] (11)  at (9.1,-1)    {2};
        \node[rectangle, draw, thick, minimum height= 3cm, minimum width = 1cm] (12)                         at (11.2,0)  {$\hdots$};
        \draw[thick]         (11.8,-0.5) -- ++(0,2) -- ++(2,0) -- ++(0,-1) -- ++(-1,0) -- ++(0,-1) -- cycle; 
        \node         (12) at (12.2,1)       {1};
        \draw[thick]         (12.8,-.5) -- ++(0,1) -- ++(1,0) -- ++(0,1) -- ++(1,0) -- ++(0,-2) -- cycle; 
        \node        (13)  at (14.2,0)      {1};
        \node[3rhor] (14)  at (13.3,-1)    {2};
        
        \draw[thick]         (14.8,-0.5) -- ++(0,2) -- ++(2,0) -- ++(0,-1) -- ++(-1,0) -- ++(0,-1) -- cycle; 
        \node         (15) at (15.2,1)       {1};
        \draw[thick]         (14.8,-1.5) -- ++(0,1) -- ++(1,0) -- ++(0,1) -- ++(1,0) -- ++(0,-2) -- cycle; 
        \node        (16)  at (16.2,-1)      {2};
        
        \node[3rhor] (17)  at (18.3,1)    {1};
        \draw[thick]         (16.8,-1.5) -- ++(0,2) -- ++(2,0) -- ++(0,-1) -- ++(-1,0) -- ++(0,-1) -- cycle; 
        \node         (18) at (17.2,0)       {2};
        \draw[thick]         (17.8,-1.5) -- ++(0,1) -- ++(1,0) -- ++(0,1) -- ++(1,0) -- ++(0,-2) -- cycle; 
        \node        (19)  at (19.2,-1)      {2};
        \node[rectangle, draw, thick, minimum height= 3cm, minimum width = 1cm] (20)                         at (20.4,0)  {$\hdots$};
        \node[3rhor] (23)  at (22.5,1)    {1};
        \draw[thick]         (21,-1.5) -- ++(0,2) -- ++(2,0) -- ++(0,-1) -- ++(-1,0) -- ++(0,-1) -- cycle; 
        \node         (21) at (21.4,0)       {2};
        \draw[thick]         (22,-1.5) -- ++(0,1) -- ++(1,0) -- ++(0,1) -- ++(1,0) -- ++(0,-2) -- cycle; 
        \node        (22)  at (23.4,-1)      {2};
        
        \node[3rhor] (23)  at (25.5,1)     {1};
        \node[3rhor] (24) [below = of 23]  {2};
        \node[3rhor] (25)  [below = of 24] {3};
        \node[3rver] (26) [right = of 24]  {$\hdots$};
        \node[3rhor] (27)  at (29.6,1)     {1};
        \node[3rhor] (28) [below = of 27]  {2};
        \node[3rhor] (29)  [below = of 28] {3};
        
        \node (29) at (1,-2) {};
        \node (30) at (5,-2) {};
        \node (31) at (9,-2) {};
        \node (32) at (13,-2) {};
        \node (33) at (15.5,-2) {};
        \node (34) at (18,-2) {};
        \node (35) at (22,-2) {};
        \node (36) at (25,-2) {};
        \node (37) at (29,-2) {};
        \node (43) at (7,-2) {};
        
        \node (38) at (3,-4) {Type 1};
        \node (39) at (11,-4) {Type 2};
        \node (40) at (16,-4) {Type 3};
        \node (41) at (21,-4) {Type 4};
        \node (42) at (27,-4) {Type 5};
        
        \draw[->, thick] (38) -- (29);
        \draw[->, thick] (38) -- (30);
        \draw[->, thick] (39) -- (31);
        \draw[->, thick] (39) -- (32);
        \draw[->, thick] (40) -- (33);
        \draw[->, thick] (40) -- (43);
        \draw[->, thick] (41) -- (34);
        \draw[->, thick] (41) -- (35);
        \draw[->, thick] (42) -- (36);
        \draw[->, thick] (42) -- (37);
        
        \draw[thick] (-0.5,-2) -- (2.4,-2);
        \draw[thick] (3.6,-2) -- (6.4,-2);
        \draw[thick] (6.7,-2) -- (7.4,-2);
        \draw[thick] (7.7,-2) -- (10.5,-2);
        \draw[thick] (11.7,-2) -- (14.6,-2);
        \draw[thick] (14.9,-2) -- (16.6,-2);
        \draw[thick] (16.9,-2) -- (19.6,-2);
        \draw[thick] (21.1,-2) -- (23.8,-2);
        \draw[thick] (24.1,-2) -- (26.8,-2);
        \draw[thick] (28.2,-2) -- (30.9,-2);
\end{tikzpicture}
}

\begin{prop}
Any $3$-ribbon tableau defined as above have content 
\[
\nu = (3\kappa_1+2\kappa_2+2\kappa_3+\kappa_4+\kappa_5, \kappa_2+\kappa_3+2\kappa_4+\kappa_5, \kappa_5),
\]
where $\nu$ is a partition of $3n$, and has Yamanouchi reading word 
\[2^{\kappa_2}3^{\kappa_5} 2^{\kappa_3+2\kappa_4+\kappa_5} 1^{3\kappa_1+2\kappa_2+2\kappa_3+\kappa_4+\kappa_5} .\]
\end{prop}

\begin{proof}
We only need to verify that $\nu_1\geq \nu_2\geq \nu_3$, which we have since $\kappa_1+\kappa_2\geq \kappa_4$, and $\nu_1+\nu_2+\nu_3 = 3(\kappa_1+\kappa_2+\kappa_3+\kappa_4+\kappa_5) = 3n$. The fact that the reading work is Yamanouchi is easily verified: the length of the sequence of $1$'s is greater than that of the following sequence of $2$'s (reading right to left) since $\kappa_4\leq \kappa_1+\kappa_2$. The sequence of $3$'s is shorter than that of the first sequence of $2$'s since $\kappa_5$ appears in its length $\kappa_3+2\kappa_4+\kappa_5$. Finally, adding the last sequence of 
$\kappa_2$ $2$'s also does not break the Yamanouchi condition.
\end{proof}

We say these $3$-ribbon tableaux are Yamanouchi $3$-ribbon tableaux of shape $(3n)^3$.

\begin{rem}
There are other ways to pave and fill $(3n)^3$ which may also give Yamanouchi words. However, the next proposition seems to confirm that our construction is good.
\end{rem}

Let $k_{i,j}$ be the number of $i$ in row $j$ of a tableau. We use this notation in the proof of the following result.

\begin{prop}
Yamanouchi $3$-ribbon tableaux of shape $(3n)^3$ and content $\nu$ are in bijection with tableaux of shape $\nu$ and content $(n)^3$ counted by the Kostka numbers $K^\nu_{(n)^3}$.
\label{thm:bijRibbTab}
\end{prop}

\begin{proof}
For a given $(\kappa_1, \kappa_2, \kappa_3, \kappa_4, \kappa_5)$, the shape $\nu$ is given, as well as $n$. Any tableau of shape $\nu$ and content $(n)^3$ has all its entries $1$ in its first row, and $\nu_3 = \kappa_5$ columns of height $3$ filled with $1,2,3$. The leftover cells are filled with the equal amount $m=n-\nu_3$ of entries $2$ and $3$, so a tableau is entirely determined by the position of those $2$'s and $3$'s in row 1 and 2. So, $\nu$ and $k_{2,2}$ are enough to determine uniquely a tableau of content $(n)^3$.  
 
Starting from a tableau of shape $\nu$ and content $(n)^3$, we can recover $(\kappa_1, \kappa_2, \kappa_3, \kappa_4, \kappa_5)$.
Let
\begin{enumerate}
\item $\kappa_1 = k_{2,1} - \lceil \frac{k_{3,2}}{2} \rceil$,
\item $\kappa_2 = k_{2,2} - \nu_3$,
\item $\kappa_3 = k_{3,2} \ (\text{mod} \ 2)$,
\item $\kappa_4 = \lfloor \frac{k_{3,2}}{2} \rfloor$,
\item $\kappa_5 = \nu_3$.
\end{enumerate}



We can easily verify that it is a composition of $n$, and that $0 \leq \kappa_3 \leq 1$. It remains to show that $\kappa_4 \leq \kappa_1 + \kappa_2$. This is equivalent to $\lfloor \frac{k_{3,2}}{2} \rfloor \leq k_{2,1} - \lceil \frac{k_{3,2}}{2} \rceil + k_{2,2} - \nu_3$, itself equivalent to $k_{3,2} + \nu_3 \leq k_{2,1}+k_{2,2} = n$. But as $k_{3,2} + \nu_3 = n - k_{3,1}$, this inequality is true.





\end{proof}

\begin{rem}
As we have seen in section~\ref{sec:tableaux}, tableaux of content $(n^3)$ having the same shape are obtained by exchanging $2$'s in row $1$ or 2 with 3's in the other row. In terms of the tuples describing Yamanouchi 3-ribbon tableaux, exchanging a $2$ in row $1$ with a $3$ in row $2$ corresponds to:
\begin{align*}
(\kappa_1,\kappa_2,\kappa_3,\kappa_4,\kappa_5) + (-1,1,1,-1,0) & \text{if } \kappa_3=0;
(\kappa_1,\kappa_2,\kappa_3,\kappa_4,\kappa_5) + (0,1,-1,0,0) & \text{if } \kappa_3=1.
\end{align*}
\end{rem}

This bijection tells us that tableaux counted by $K_{(n)^3}^{\nu}$ are in bijection with compositions of $n$ in $5$ parts $(\kappa_1, \kappa_2, \kappa_3, \kappa_4, \kappa_5)$ that are solutions to a certain equation system.
	
We study these compositions as integer points on polytopes in the next section. We denote the set of solutions (for $\nu$ fixed) by $\text{Yam}_3(3(n^3),\nu)$, and the set of all possible solutions (for any $\nu\vdash 3n$) by $\text{Yam}_3(3(n^3))$. We generally identify these solutions, the corresponding Yamanouchi $3$-ribbon tableaux and tableaux of content $(n)^3$ further on.

\section{Integer points on polytopes}
\label{section:polytopes}

The compositions 
$(\kappa_1,\kappa_2,\kappa_3,\kappa_4,\kappa_5)\in Yam_3((3n)^3)$ can be seen as positive integer points on the polytope $\mathbb{P}_n$ defined by the following equations.
\begin{itemize}
\item $\kappa_1+\kappa_2+\kappa_3+\kappa_4+\kappa_5 = n$, 
\item $\kappa_i\geq 0$,
\item $\kappa_3\leq 1$,
\item $\kappa_1+\kappa_2\geq \kappa_4$.
\end{itemize}

Each point with $\nu = (3\kappa_1+2\kappa_2+2\kappa_3+\kappa_4+\kappa_5, \kappa_2+\kappa_3+2\kappa_4+\kappa_5, \kappa_5)$ is in $Yam_3((3n)^3,\nu)$, and corresponds to a tableau counted by $K_{(n)^3}^\nu$, 
the one 
with $\kappa_2 + \kappa_5$ $2$'s on its second row. 

\subsection{Visualizing $\mathbb{P}_n$}
The polytope $\mathbb{P}_n$ lies in a $5$ dimension space. In order to visualize it, we consider the projection onto $(\kappa_1, \kappa_2, \kappa_4)$ with :
\begin{itemize}
\item $\kappa_1+\kappa_2+\kappa_4 = m$, 
\item $\kappa_i\geq 0$,
\item $\kappa_1+\kappa_2\geq \kappa_4$.
\end{itemize}
This corresponds to the restriction of $\mathbb{P}_n$ with $\kappa_5 = n-m$ and $\kappa_3=0$.
We have seen that compositions of $n$ with a fixed $\kappa_5>0$ are in bijection with the compositions of $m=n-\kappa_5$ (corresponding to removing columns of height 3), so we do not lose information by considering this projection, because we know that adding such a column conjugates the type.

Moreover, when $\kappa_3=1$, we can still consider the points on the same object, since the points simply lie in the "inner" polytope with $\kappa_1+\kappa_2+\kappa_4=m-1$.\\

Therefore we can build the union of polytopes $P_m$ consisting of all points of $\mathbb{P}_n$ such that $\kappa_5 = n-m-\kappa_3$, with either $\kappa_3=0$ or $\kappa_3=1$. All points of $\mathbb{P}_n$ occur this way, so visualizing all the $P_m$, for $0\leq m \leq n$, give us a good understanding of $\mathbb{P}_n$. Further on, we oversimplify the situation by calling these $P_m$ polytopes.\\ 

See figure~\ref{fig:P_n} for the first $P_m$, for $m=0,1,2,3,4,5,6$. Together, they give projections of $\mathbb{P}_6$ at $\kappa_5 = 6,5,4,3,2,1,0$. Figure~\ref{fig:P_nSide} illustrates the fact that $P_m$ is the union of two polytopes.  \\

\begin{figure}
\hspace{-4em}
\begin{minipage}{0.4\textwidth}
\centering
\includegraphics[width=0.99\linewidth]{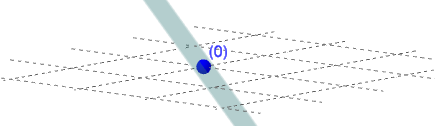}
\end{minipage}%
\begin{minipage}{0.4\textwidth}
\centering
\includegraphics[width=0.99\linewidth]{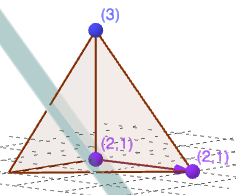}
\end{minipage}%
\begin{minipage}{0.4\textwidth}
\centering
\includegraphics[width=0.99\linewidth]{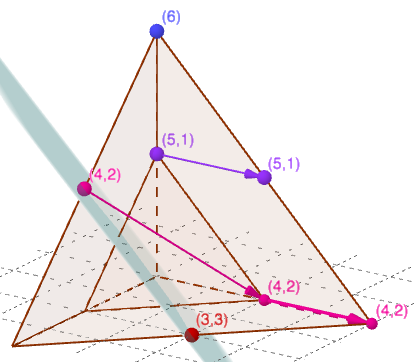}
\end{minipage}

\vspace{-2em}
\hspace{-4em}
\begin{minipage}{0.4\textwidth}
\centering
\includegraphics[width=0.99\linewidth]{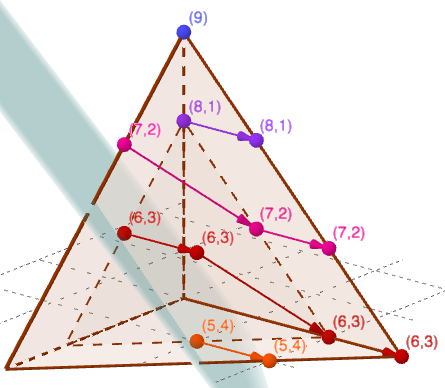}
\end{minipage}%
\begin{minipage}{0.4\textwidth}
\centering
\includegraphics[width=0.99\linewidth]{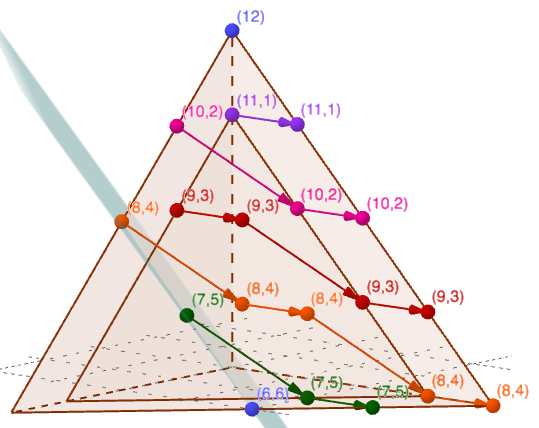}
\end{minipage}%
\begin{minipage}{0.4\textwidth}
\centering
\includegraphics[width=0.99\linewidth]{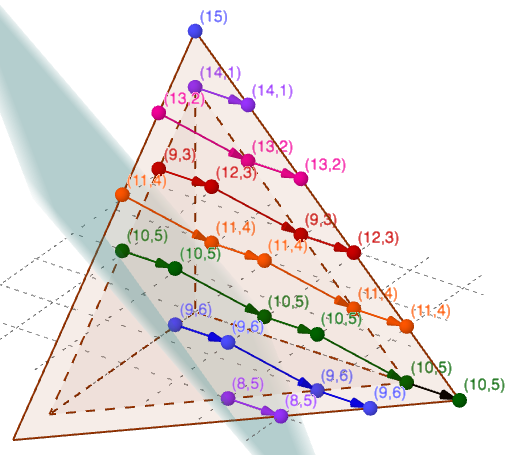}
\end{minipage}%

\vspace{-1.5em}
\hspace{11em}
\begin{minipage}{0.45
\textwidth}
\centering
\includegraphics[width=0.99\linewidth]{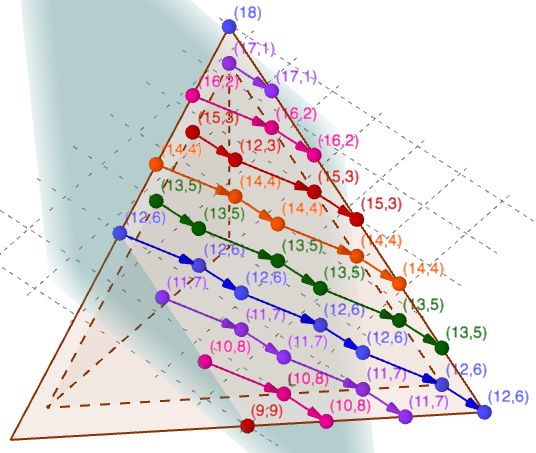}
\end{minipage}%
\begin{minipage}{0.35
\textwidth}
\centering
\includegraphics[width=0.99\linewidth]{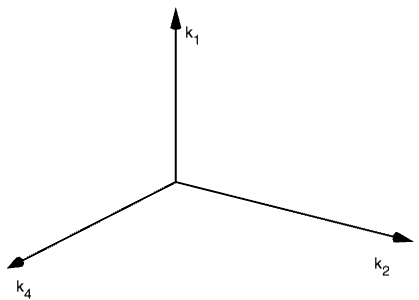}
\end{minipage}%

\caption{Projections $P_0, P_1, P_2, P_3, P_4, P_5, P_6$ of $\mathbb{P}_6$ at $\kappa_5= 6,5,4,3,2,1,0$. 
The shaded plane is $\kappa_1+\kappa_2=\kappa_4$.
The partition strands are colored identically from one figure to the next if they occur as a shift of $\kappa_1$ by $1$. This allows to visualize how $P_i\subset P_{i+1}$.  }
\label{fig:P_n}
\end{figure}

\begin{figure}
	\centering
		\includegraphics[width=0.2\linewidth]{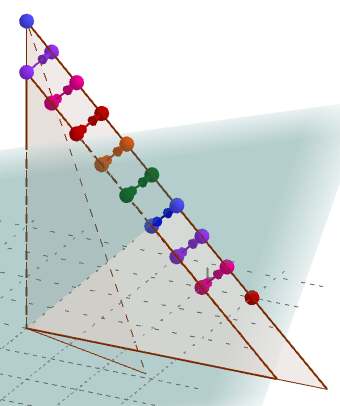}

	\caption{Side cut of $P_6$ which illustrates the fact that $P_m$ is the union of two polytopes, an inner one with $\kappa_3=1$, and an outer one with $\kappa_3=0$. Integer points of both polytopes of the same $j$-strand lie on the line of direction $(-1,+2,-1)$, by the geometry of transformations.}
	\label{fig:P_nSide}
\end{figure}

In the appendix, we give a closed formula for the number of integer points in both $P_m$ and $\mathbb{P}_n$. This is of interest, since they respectively give the number of tableaux of content $(n)^3$ with at most two rows, and the number of tableaux of content $(n)^3$ of any shape.\\



We can focus on tableaux of content $(n)^3$ with at most two parts, the ones who appear in $P_n$. 
A partition of $(3n)$ with at most two parts is of the form $(3n-j,j)$, for $0 \leq j \leq \lfloor \frac{3n}{2} \rfloor$. We call the collection of points corresponding to tableaux of that shape in $P_n$ the $j$-\textit{strand}. They correspond to points $(\kappa_1,\kappa_2,\kappa_3,\kappa_4)$ such that $\kappa_2+\kappa_3+2\kappa_4 = j$. 

As we have seen in section~\ref{section:ribTab}, the transformations between points which preserve the associated partition are $(\kappa_1,\kappa_2,\kappa_3,\kappa_4)\pm(-1,+1,+1,-1)$ or $(\kappa_1,\kappa_2,\kappa_3,\kappa_4)\pm(0,+1,-1,0)$, depending on the value of $\kappa_3$. Therefore the points giving $(3n-j,j)$ are linked by these transformations, and form an actual strand.
These transformations correspond precisely to exchanging a $2$ on the first row with a $3$ on the second row of the associated tableaux.

\subsection{Inclusion of polytopes, complete and incomplete strands}

\begin{prop} The polytope $P_{n-1}$ is included in $P_n$ for $n\geq 1$, with the strands of $\nu\vdash 3(n-1)$ becoming that of $\nu + (3) = (\nu_1+3,\nu_2,\nu_2)\vdash 3n$.
\label{prop:PolytopeInclusion} \end{prop}

\begin{proof}
This injection corresponds to the operation $T\vee \young(123)$ defined in section~\ref{section:general}. This gives a shift to all points of $P_{n-1}$ by adding 1 to $\kappa_1$.  
\end{proof}


%

A $j$-strand is called complete when 
this injection of $P_n$ into $P_{n+1}$ does not modify the number of tableaux it contains. The complete $j$-strands of $P_n$ are associated to partitions $\nu=(3n-j,j)$ with $j\leq n$, and correspond to the case where $K_{(n)^3}^{\nu} 
= j+1$. The initial point of a complete $j$-strand corresponds to the tableau which only have $3$'s on the second row, 
and its final point, to the one which only have $2$'s on the second row.\\

For $n<j\leq \lfloor\frac{3n}{2}\rfloor$, the $j$-strands of $P_n$ are incomplete. There are three tableaux in the $j$-strand of $P_{n+1}$ which are not obtained from tableaux of $P_{n}$ by adding three cells with filling $1,2,3$ in their first row: the first tableau of the $j$-strand which has no three on it's first row, the last one which has no two's on it's first row, and the previous to last one, which becomes non standard when removing three cells with filling $1,2,3$ in its first row. See figure~\ref{fig:ReductionTab3cellules1eLigne}.\\ 

Going from $P_n$ to $P_{n-1}$ either preserve the $j$-strands (if $j\leq n$) or removes three tableaux to them: one at the beginning of each $j$-strand, and two at their end. The inclusion of $P_{n-1}$ in $P_n$ also justifies the notation of $j$-strand, since it doesn't modify the second part $j$ of the associated partitions, the strands depend only on the values of $j$ and $n$.

\begin{rem}
    The three tableaux which are removed from the $j$-stand when going from $P_n$ to $P_{n-1}$ are removed because they do not have the form $\bar{T} \vee \young(123)$.
\end{rem}

\begin{figure}[h!]
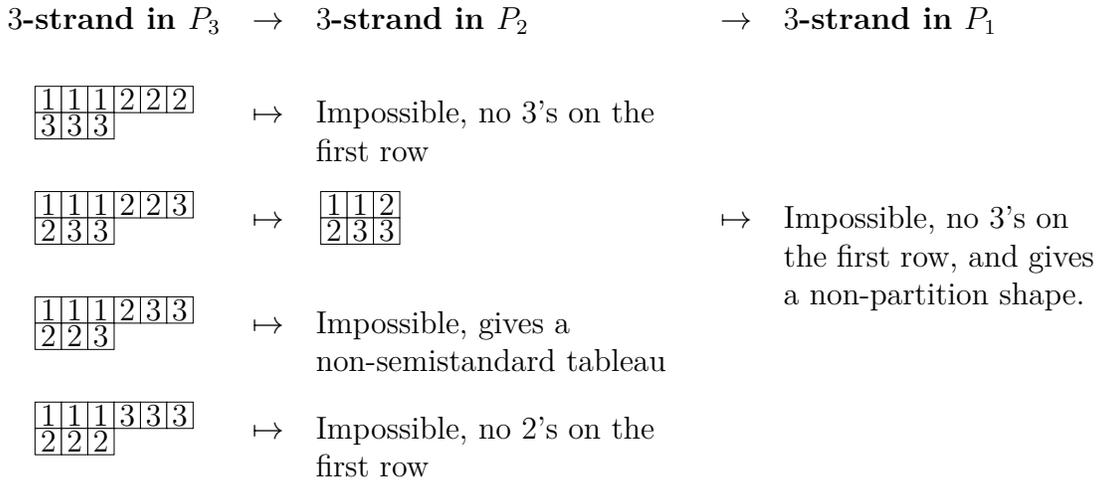

	\vspace{-1em}
	\begin{center}
		\begin{tabular}{cclcl}
			\textbf{$3$-strand in $P_3$} & $\rightarrow $ & \textbf{$3$-strand in $P_2$} & $\rightarrow $ & \textbf{$3$-strand in $P_1$}\\
			&&&&\\
			\young(111222,333) & $\mapsto$ & \multirow{2}{12em}{Impossible, no $3$'s on the first row}&&\\
			&&&&\\
			\young(111223,233) & $\mapsto$ & \young(112,233)& $\mapsto$ & \multirow{3}{10em}{Impossible, no $3$'s on the first row, and gives a non-partition shape.}\\
			&&&&\\
			\young(111233,223) & $\mapsto$ & \multirow{2}{12em}{Impossible, gives a non-semistandard tableau}&&\\
			&&&&\\
			\young(111333,222) & $\mapsto$ & \multirow{2}{12em}{Impossible, no $2$'s on the first row}&&\\
			&&&&\\
		\end{tabular}
	\end{center}
	\vspace{-1em}
	\caption{Reduction of the $3$-strand from $P_3$ to $P_1$. The $3$-strand is complete in $P_3$, and counts $4$ tableaux. The $3$-strand in $P_2$ counts only $1$ tableau, and it vanishes in $P_1$.}
	\label{fig:ReductionTab3cellules1eLigne}
\end{figure}

\begin{prop}
	An (incomplete) $j$-strand first appears in $P_{\lceil \frac{2j}{3}\rceil}$ 
	, grows until complete, and remains complete for $P_n$ with $n\geq j$.
	It holds respectively $1$, $2$ or $3$ tableaux in its first occurrence, depending on whether $(j\mod 3) = 0, (j\mod 3) = 1$ or $(j\mod 3) = 2$, the first one being such that $k_{3,2}=\lceil \frac{2j}{3}\rceil$ and $k_{2,2}=\lfloor \frac{j}{3}\rfloor$.
\end{prop}

\begin{proof}

Tableaux in different $P_n$ which have the same second row correspond to the same points of the $j$-strand, using the inclusion. In a complete $j$-stand, the second row of the first tableau has no $2$'s. After $\lfloor \frac{j}{3}\rfloor$ exchanges, the second row holds $\lfloor \frac{j}{3}\rfloor$ $2$'s and $\lceil \frac{2j}{3} \rceil$ $3$'s, so this is the second row of the first tableau of the $j$-strand at its first occurrence.
\end{proof}

We have seen in section~\ref{section:general} that adding three cells with entries $1,2,3$ to the first row does not change the type. If a tableau first appears in $P_n$, we can keep the same type for the corresponding tableau in $P_{N}$ for $N > n$.

\section{Types for tableaux of content $(n)^3$}\label{section:Types}

We can use results of previous sections to define types for points in $\mathbb{P}_n$. We say $(\kappa_1,\kappa_2,\kappa_3,\kappa_4,\kappa_5)\in Yam_3{((3n)^3)} \subset \mathbb{P}_n$ is of type $\nu$ if and only if the corresponding tableau has type $\nu$.\\

By proposition~\ref{prop:columns}, we know that the type of $(\kappa_1,\kappa_2,\kappa_3,\kappa_4,\kappa_5)$ is:

\begin{itemize}
    \item The same as the one of $(\kappa_1,\kappa_2,\kappa_3,\kappa_4,0)$ if $\kappa_5$ even;
    \item The conjugate of the one of $(\kappa_1,\kappa_2,\kappa_3,\kappa_4,0)$ if $\kappa_5$ odd.
\end{itemize}

We can now consider the type conditions the points of $P_n$ must satisfy to agree with the results of previous sections. Proposition~\ref{prop:PolytopeInclusion} tells us we can restrict ourselves to $P_n$ in $\mathbb{P}_n$.\\ 

Certain type attributions are determined, but others require choices. We discuss both, and highlight which choices appear to be the simplest and should therefore be used.

\subsection{Determined type attributions}

The next lemma tells us about $2-$subtypes. Recall that $\young(123)$ and $\young(12,3)$ have $2$-subtype $\young(12)$, and $\young(13,2)$ and $\young(1,2,3)$ have $2$-subtype $\young(1,2)$.

\begin{lem}
Along a $j$-strand of $P_n$, the types have alternatingly $2$-subtype $\young(12)$ and $\young(1,2)$.

	
The first tableau of any \emph{complete} $j$-stands of $P_n$ ($j\leq n$) has $2$-subtype $\young(12)$, 
and its last tableau, $2$-subtype
\vspace{-1.5em}
\begin{center}
$\left\{
\begin{tabular}{cc}
$\young(12)$ & if $j$ even\\
$\young(1,2)$ & if $j$ odd\\
\end{tabular}\right.$.
\end{center}

The first tableau of any \emph{incomplete} $j$-stands of $P_n$ ($j> n$) has $2$-subtype
\begin{center}
$\left\{
\begin{tabular}{cc}
$\young(12)$ & if $j\equiv_2 n$ \\
$\young(1,2)$ & if $j\not\equiv_2 n$\\
\end{tabular}\right.$,
\end{center}
and its last tableau, $2$-subtype
\vspace{-2em} 
\begin{center}
$\left\{
\begin{tabular}{cc}
$\young(12)$ & if $j$ even\\
$\young(1,2)$ & if $j$ odd\\
\end{tabular}\right.$.
\end{center}

\end{lem}

\begin{proof}
The alternance between $2$-subtypes follows from the fact 
that the parity of the number of $2$'s on the second row of tableaux determines the $2$-subtype.\\
	
In the case of complete $j$-strands, the first tableau have no entry $2$ on its second row. 
Then its $2$-subtype is \young(12).
For the last tableau, it has $j$ entries $2$ on it's second row. Therefore its $2$-subtype depends only on the parity of $j$.\\

In the case of incomplete $j$-strands, the first tableau has $j-n$ entries $2$ on its second row. Therefore, the number of $2$'s on the second row is even when $j$ and $n$ have the same parity, and odd otherwise, which gives us the $2$-subtypes as noted.
The last tableau has $\kappa_2 = 2n-j$ entries $2$ on its second row, so its $2$-subtype depends only on the parity of $j$.
\end{proof}

Using the inclusion of proposition~\ref{prop:FirstRow}, we see that a plethystic decomposition of $h_n^3$ for all $n$ is equivalent to coloring the points of strands of $P_\infty = \bigcup_{n=0}^{\infty} P_n$ (by their type). Let's start by examining the tableaux of $P_n$ which have $\kappa_2=\kappa_4$. These lie in the "center" of the $P_n$, at the junction of the planes $\kappa_2+\kappa_4=n$ and $\kappa_4 = \kappa_1+\kappa_2$ which mark the frontiers between $P_{n}$ and $P_{n+1}$. These tableaux are the first to appear in a new $j$-strand.

\begin{lem}
	In order to agree with Chen's rule and proposition~\ref{prop:FirstRow}, a type is determined for the integer points (tableaux) of $P_n$ on the plane $\kappa_2=\kappa_4$,  according to the value of the second part $j$ of the associated partition $(3n-j,j)$:
	\begin{center}
	$\left\{
	\begin{tabular}{cl}
	$\young(123)$ &if $j\equiv_6 0$\\
	$\young(12,3)$  & if $j\equiv_6 1$\\
	$\young(1,2,3)$ & if $j\equiv_6 3$\\
	$\young(13,2)$ & if $j\equiv_6 4$
	\end{tabular}.\right.$
	\end{center}
\label{prop:attributionMiddlePn}
\end{lem}

\begin{proof}
	There are two points in $P_n$ which have $\kappa_2=\kappa_4 = i$ for a certain $i\in\N$: $(\kappa_1,\kappa_2,\kappa_3,\kappa_4) \in \{ (n-2i,i,0,i), (n-2i-1,i,1,i) \}$. The first lies in the strand associated to the partition $(3n-3i,3i)$, and the second, to that associated to $(3n-3i-1,3i+1)$. Therefore, the only strands with a point on the plane $\kappa_2=\kappa_4$ have second part $j\equiv_3 \{0,1\} \equiv_6 \{0,1,3,4\}$.\\
	
	Recall that a $j$-strand appears first in $P_n$ for $n=\lceil \frac{2j}{3} \rceil$. 
	
	If $j\equiv_6 \{0,3\}$, $K^{(3n-j,j)}_{(n)^3} = 1$. In the first case, $j$ is even, so the unique tableau in the incomplete $j$-strand must be of type $\young(123)$.
	In the second case, $j$ is odd, so the unique tableau in the incomplete $j$-strand must have type $\young(1,2,3)$. This agrees with using the parity of the number of $2$'s on the second row.
	We have $K^{(3n-j,j)}_{(n)^3} = 1$ only if $\min(3n-2j,j) = 0$, so the two parts of $(3n-j,j)$ must have the same length $j=\frac{3n}{2}$, and $n$ must be even.
	The second row then holds $\frac{3n}{2} - n = \frac{n}{2}$ entries $2$, which has the same parity as $j$.\\
	
	For $j\equiv_6 \{1,4\}$, $K^{(3n-j,j)}_{(n)^3} = 2$, so the second part $j$ must be one less than the first part: $j=\frac{3n-1}{2}$, and $n$ must be odd. Therefore the two tableaux in the incomplete $j$-strand must have type $\young(12,3)$ or $\young(13,2)$. The first tableau of the $j$-strand has $(j-1)$ $2$'s on the second row, and the second, $j$ $2$'s. Note that only the first rests on the plane $\kappa_2=\kappa_4$.\\
	
	In the case $j\equiv_6 1$, $j$ is odd, so the first tableau must be associated to $\young(12,3)$, and the second to $\young(13,2)$. In the case $j\equiv_6 4$, $j$ is even, so it is the opposite. 
\end{proof}

\subsection{Choice of order for undetermined type attributions}
The new $j$-strands of $P_n$ 
 have $1$ ($j\equiv_3 0$) or $3$ ($j\equiv_3 2$) tableaux if $n$ even, and $2$ ($j\equiv_3 2$) tableaux if $n$ odd. If a new $j$-strand holds $1$ or $2$ tableaux, their types are entirely determined as we have seen above. 
If it holds three tableaux, the type of the middle one is determined: $\young(13,2)$ if $j$ even, $\young(12,3)$ if $j$ odd. For the first and last attribution of the $j$-strand, however, 
the two types are determined, but not which is attributed to the first or last tableau.\\ 

Moreover, going from $P_n$ to $P_{n+1}$ adds three tableaux to all incomplete $j$-strands, one to their start, and two to their end, until they become complete. At each step, the type of one of the two tableaux added at the end of a $j$-strand is determined, the two other types are also determined, but not the order of the attribution.\\ 

The type of the tableau added at the start of the $j$-strand is always undetermined, so fixing a rule for the attribution of these types determines the others. Any pattern which starts with either $\young(123)$ or $\young(12,3)$, alternates between both $2-$subtypes $\young(12)$ and $\young(1,2)$, and ends with the wanted "middle" attribution, at $\kappa_2=\kappa_4$, gives a type attribution which agrees with Chen's rules and our rules.\\ 

There are infinitely many possible patterns which agree with the above conditions, however some reveal interesting patterns. We now discuss one such type attribution, and argue that it is the simplest, and most elegant, one.\\

\subsection{Simplest type attribution}

This construction rests on the choice that all first type attributions of a $j$-strand are in turn $\young(123)$ or $\young(1,2,3)$. It was found 
independantly by Mike Zabrocki and 
coauthors \cite{Zabrocki}.

\begin{theo*}
	Let $T$ be a tableau of shape $(3n-j,j)$ with $\kappa_2$ entries $2$, and $k_{3,2}$ entries $3$, on its second row. Define the type of $T$  to be
	\begin{center}
		$\left\{
		\begin{tabular}{clcc}
		$\young(123)$ & if $\kappa_2$ even & AND & $\kappa_2<\left\lfloor\frac{k_{3,2}}{2}\right\rfloor$ , OR $\kappa_2=\left\lfloor\frac{k_{3,2}}{2}\right\rfloor$ AND $j\equiv_3 0$\\			$\young(12,3)$ & if $\kappa_2$ even & AND & $\kappa_2>\left\lfloor\frac{k_{3,2}}{2}\right\rfloor$ , OR $\kappa_2=\left\lfloor\frac{k_{3,2}}{2}\right\rfloor$ AND $j\equiv_3 1$\\
		$\young(1,2,3)$ & if $\kappa_2$ odd & AND & $\kappa_2<\left\lfloor\frac{k_{3,2}}{2}\right\rfloor$ , OR $\kappa_2=\left\lfloor\frac{k_{3,2}}{2}\right\rfloor$ AND $j\equiv_3 0$\\
		$\young(13,2)$ & if $\kappa_2$ odd & AND & $\kappa_2>\left\lfloor\frac{k_{3,2}}{2}\right\rfloor$ , OR $\kappa_2=\left\lfloor\frac{k_{3,2}}{2}\right\rfloor$ AND $j\equiv_3 1$
		\end{tabular}\right.$.
	\end{center}

The number of tableaux of each type coincides with Chen's formulas. Explicitly, the coefficient of $s_\nu$ in the plethysm $s_{shape(t)}[h_n]$ indexed by the standard tableau $t$ is given by $c_\nu^{(t)}=| \{T\in \SSYT(\nu,(n)^3) \ | \ \text{type}(T) = t\ \} |$.

\label{Conj:ColoringS3S111}
\end{theo*}

This gives the coloring of figure~\ref{fig:ColoringS3S111}. 
It appears visually to be the simplest type attribution. 


\begin{figure}[h]
	
	\begin{center}	\includegraphics[width=0.4\linewidth]{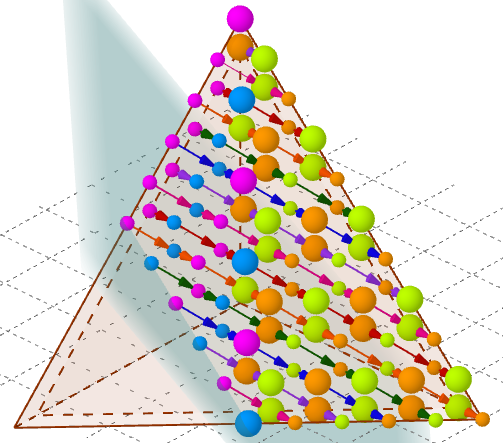}
	\end{center}
	\vspace{-1em}
	\caption{Coloring of the points of $P_{10}$ according to theorem~\ref{Conj:ColoringS3S111}, where pink points are of type $\protect\young(123)$, orange points, of type $\protect\young(12,3)$, green points, of type $\protect\young(13,2)$ and blue points, of type $\protect\young(1,2,3)$. 
	Larger points indicate points with determined type.}
	\label{fig:ColoringS3S111}
\end{figure}

\begin{rem}
	Since $\kappa_4 = \left\lfloor \frac{k_{3,2}}{2} \right\rfloor$, the conditions account to verifying whether $\kappa_2<\kappa_4$ or $\kappa_2>\kappa_4$, and if $\kappa_2=\kappa_4$, what is the value of $j$ modulo $6$. 
\end{rem}

\begin{proof}
	The complete $j$-strands have $K^{(3n-j,j)}_{(n)^3} = j+1$ tableaux. 
	The first tableau of a $j$-strand has only $3$'s on its second row, and each exchange increases the number of $2$'s on the second row by $1$, so the last tableau has only $2$'s on its second row.\\ 
	
	The condition $\kappa_2 < \left\lfloor \frac{k_{3,2}}{2} \right\rfloor$ is valid for $\kappa_2<\left\lceil \frac{j}{3}\right\rceil $, since $k_{3,2}=j-\kappa_2$, so the first $\left\lceil \frac{j}{3} \right\rceil$ tableaux have alternatingly type $\young(123)$ and $\young(1,2,3)$. 
	The condition $\kappa_2> \left\lfloor \frac{k_{3,2}}{2} \right\rfloor$ is valid for $\left\lfloor \frac{j}{3} \right\rfloor +1\leq \kappa_2 \leq j$, so the last $\left\lfloor \frac{2j}{3} \right\rfloor$ tableaux have alternatingly type $\young(12,3)$ and $\young(13,2)$.\\
	
	Let's verify we get the exact cardinalities obtained through Chen's rules for $j= 6k+\ell$, $\ell \in \{ 0,1,2,3,4,5 \}$. Suppose $j\equiv_6 0$, so $K^{(3n-j,j)}_{(n)^3} = 6k+1$.
	We have seen that the $j$-strand has one point on the line $\kappa_2=\kappa_4$ if $j \equiv_3 \{ 0,1 \}$, with attributions given in proposition~\ref{prop:attributionMiddlePn}.

	Then the "middle attribution" here is $\young(123)$, which gives $k+1 = $
	$\left\lceil \frac{j+1}{6} \right\rceil $ tableaux of type \young(123), $k = \left\lfloor \frac{j+1}{6} \right\rfloor$ tableaux of type $\young(1,2,3)$, and $2k$ tableaux of each type $\young(12,3)$ and $\young(13,2)$. This gives exactly the cardinalities obtained through Chen's rule for $K^\nu_{(n)^3}\equiv_6 1$.
	%
	%
	The proof that this is true for the other complete $j$-strands 
	is done using exactly the same reasoning.\\
	
	Let's now see what happens when $j>n$ and the $j$-strand is incomplete in $P_n$. The number of tableaux in the $j$-strand is then $K^{(3n-j,j)}_{(n)^3} = 3n-2j+1$. 
	We start by considering the types of the tableaux in the first occurrence of a $j$-strand, verify it satisfies Chen's rule, and verify that adding tableaux according to the attributions above also does.
	
	At its first occurrence, a $j$-strand counts either $1$ ($j\equiv_3 1$), $2$ ($j\equiv_3 2$) or $3$ ($j\equiv_3 2$) tableaux. The attributions described in the proof of proposition~\ref{prop:attributionMiddlePn} agree with Chen's rule, and with the parity of the number of $2$'s on the second row.
	
	Going from $P_n$ to $P_{n+1}$ adds one tableau to the start of the strand with type either $\young(123)$ or $\young(1,2,3)$, and two tableaux to its end, of type $\young(12,3)$ and $\young(13,2)$ (in one order or the other).
	When starting with one tableau, 
	it gives four tableaux, one for each plethysms. When starting with two tableaux, 
	it gives five tableaux, of type either $\young(123)$ ($j$ even) or $\young(1,2,3)$ ($j$ odd) and two times $\young(12,3)$ and $\young(13,2)$. When starting with three tableaux, 
	it gives six tableaux, one of each type $\young(123)$ and $\young(1,2,3)$ and two of each type $\young(12,3)$ and $\young(13,2)$. This 
	agrees with Chen's rule and with the rules for $2$-subtypes, since the attributions alternate between the two $2$-subtypes.\\
	
	Going from $P_{n}$ to $P_{n+2}$ then adds the six types above, 
	and gives back the same number of tableaux modulo $6$ as in the first occurrence of the $j$-strand. Since these attributions
	agree with Chen's rule (and with the rules for $2$-subtypes), 
	then all attributions do.  
\end{proof}

\begin{rem}
	This construction is particularly interesting when considering the antidiagonals appearing in $P_n$: the tableaux in different strands that hold the same number of entries $2$ on their second row. They must all have the same $2$-subtype.
    In a specific antidiagonal, the attributions are segregated in two blocks, again before and after passing the threshold plane $\kappa_2=\kappa_4$: when the number of $3$'s on the second line is greater to two times the number of $2$ (but not a difference of $1$), the type is either $\young(123)$ or $\young(1,2,3)$ (depending on the $2$-subtype), and when it is smaller than two times the number of $2$'s, the type is either $\young(12,3)$ or $\young(13,2)$. For the point with $\kappa_2=\kappa_4$ (exactly two points on each complete antidiagonal), the types are exchanged. This means that when constructing a tableau of content $(n)^3$, there is a threshold when adding $3$'s to the second line after which the plethystic attribution change, and this threshold occurs when about a third of the second row holds $2$'s.\\ 
    This also points towards a certain "stability" condition under the operation $\vee \young(12,3)$, where it preserves the plethystic attribution up to a threshold point, after which it is stable again. Proving this stability condition whould then show that the attribution described above is the unique one, as it is the only one with this property.
\end{rem}

\subsection{Other type attributions}

Another possible type attribution rests on the choice that all first type attributions are alternatingly $\young(12,3)$ and $\young(13,2)$. This construction was first found by Mike Zabrocki. The resulting coloring is illustrated in figure~\ref{fig:ColoringS21S21}. 
There are infinitely many other 
plethystic attributions which also respect the rules discussed above.

\begin{figure}[h!]

	\begin{center}	\includegraphics[width=0.35\linewidth]{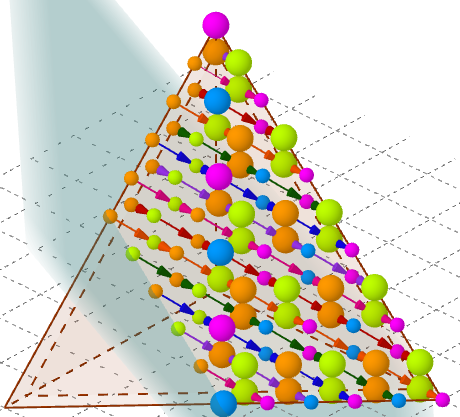}
	\end{center}
	\vspace{-1em}
	\caption{Another coloring of the points of $P_{10}$, 
	where pink points are of type $\protect\young(123)$, orange points, of type $\protect\young(12,3)$, green points, of type $\protect\young(13,2)$ and blue points, of type $\protect\young(1,2,3)$. Large points indicate determined plethystic attributions.}
	\label{fig:ColoringS21S21}
\end{figure}

\section{Generalizing to $e_n^m$}\label{section:elementary}

All results above can be generalized to $e_n^3$. To do so, first note that:

\[
e_n^m = \omega(h_n)^m = \displaystyle \sum_{\mu \vdash nm} K_{(n)^m}^{\mu'} s_{\mu}.
\]

The Kostka numbers $K_{(n)^m}^{\mu'}$ can be interpreted as the number of \textit{conjugate tableaux} of shape $\mu$ and content $(n)^m$, which are tableaux such that the entries are strictly increasing along the rows and weakly increasing along the columns. This is because each of them 
are obtained by conjugating a tableau of shape $\mu$ (\textit{i.e.} reflecting it along the main diagonal).\\

Using proposition~\ref{Plethystic types}, we also have that

\[
e_n^m = \displaystyle \sum_{\nu \vdash m} f^{\nu} s_{\nu}[e_n].
\]

We can understand the plethystic decomposition of $e_n^m$ by attributing types to conjugate tableaux. By using the $\omega$ involution introduced in section~\ref{section:plethysm}, we have that

\[
\langle s_{\mu}, s_{\nu}[e_n] \rangle = \langle \omega(s_{\mu}), \omega(s_{\nu}[e_n]) \rangle = 
\begin{cases}
\langle s_{\mu'}, s_{\nu}[h_n] \rangle & \text{if $n$ is even}; \\
\langle s_{\mu'}, s_{\nu'}[h_n] \rangle & \text{if $n$ is odd}.
\end{cases}
\]

So, we have that the type of a conjugate tableau is:

\begin{itemize}
\item The same type as its conjugate tableau if $n$ is even;
\item The conjugate type of its conjugate tableau if $n$ is odd.
\end{itemize}

Using this, it means that our results for the plethystic decomposition of $h_n^m$ also gives the plethystic decomposition of $e_n^m$. So, the previous results for $h_n^3$ are easily extended to $e_n^3$.

\section{Plethysm of a product of symmetric functions}\label{section:Product}

Let $f_1,\ldots,f_k$ be symmetric functions. We investigate how one can understand the plethystic decomposition of $(f_1\ldots f_k)^m$ using the knowledge of the plethystic decomposition of each $f_i^m$. It turns out that it involves an understanding of the Kronecker coefficients, which we don't have in general. For small values of $m$, we can compute them, and doing so, we give explicit results for $m \leq 3$. We then apply results from $\cite{MaasTetreault}$ to show how to apply this technique to $h_{\lambda}^m$ and $e_{\lambda}^m$. 
We then get an explicit plethystic decomposition of $h_{\lambda}^3$ and $e_{\lambda}^3$.

\subsection{The general case}
\label{section:Kronecker}

Consider the symmetric function $(f_1\ldots f_k)^m$. There are two possible plethystic decompositions, depending on whether we decompose the whole product or each function separately.

\begin{align*}
(f_1 \ldots f_k)^m  &= \displaystyle \sum_{t \in \SYT_m}                                       s_{\text{shape}(t)}[f_1 \ldots f_k] \\
(f_1 \ldots f_k)^m  &= f_1^m \ldots f_k^m \\
                    &= \prod_{i=1}^k \left( \displaystyle \sum_{t_i \in \SYT_m} s_{\text{shape}(t_i)}[f_i] \right) \\
                    &= \displaystyle \sum_{(t_1, \ldots, t_k) \in \SYT_m^k} \left( \prod_{i=1}^k s_{\text{shape}(t_i)}[f_i] \right).
\end{align*}

In our case, we know the latter decomposition, and we want to recover the former. That is, we would like to find a map $F_{m,k} : \SYT_m^k \rightarrow \SYT_m$ such that 

\[
s_{\text{shape}(t)}[f_1 \ldots f_k] = \displaystyle \sum_{(t_1, \ldots, t_k) \in F_{m,k}^{-1}(t)} \left( \prod_{i=1}^k s_{\text{shape}(t_i)}[f_i] \right).
\]

We call such a map a \textit{Kronecker map}. We can construct one recursively on $k$. When $k=2$, we use the following result, which can be found in \cite{Macdonald} (chapter 1, equation 7.9). 




\begin{prop}\label{Kronecker}
For $\mu \vdash m$ and $f_1,f_2$ two symmetric functions. Then:
\[
s_{\mu}[f_1f_2] = \displaystyle \sum_{\alpha, \beta \vdash m} g_{\alpha,\beta}^{\mu} s_{\alpha}[f_1]s_{\beta}[f_2],
\]
where $g_{\alpha,\beta}^{\mu}$ are the \textit{Kronecker coefficients}. 
\end{prop}

These coefficients are the coefficients of $s_{\mu}$ in the \textit{internal product} $s_{\alpha} \ast s_{\beta}$, which is defined in \cite{Macdonald} (chapter 1, section 7).\\ 

This result tells us that the map $F_{m,2}$ must be such that the number of pairs $(t_1,t_2) \in F_{m,2}^{-1}(t)$, where $t_1$ is of shape $\alpha$ and $t_2$ is of shape $\beta$, is equal to $g_{\alpha,\beta}^{\text{shape}(t)}$. This is the reason why we call $F_{m,k}$ a Kronecker map. Kronecker coefficients have no nice combinatorial description, so constructing such a map for all $m$ seems unlikely. However, it is doable for small values of $m$. There are usually many choices for a Kronecker map, because Kronecker coefficients only tells us about the shape of the pairs of tableaux, not the tableaux themselves.\\

When $m=2$, there is only one choice of Kronecker map, which is:

\begin{align*}
F_{2,2}^{-1}\left( \young(12) \right) &= \left\{ \left( \young(12), \young(12) \right), \left( \young(1,2) , \young(1,2) \right) \right\}; \\
F_{2,2}^{-1}\left( \young(1,2) \right) &= \left\{ \left( \young(12), \young(1,2) \right), \left( \young(1,2) , \young(12) \right) \right\}.
\end{align*}


When $m=3$, there are many possible choices. The map below is such that if $(t_1,t_2) \in F_{2,3}^{-1}(t)$, then $(t_1\downarrow_2,t_2\downarrow_2) \in F_{2,2}^{-1}(t\downarrow_2)$. 
However this choice still doesn't determine all pairs.

\begin{align*}
F_{3,2}^{-1}\left( \young(123) \right) &= \left\{ \left( \young(123), \young(123) \right), \left( \young(12,3), \young(12,3) \right), \left( \young(1,2,3) , \young(1,2,3) \right) \right\}; \\
F_{3,2}^{-1}\left( \young(12,3) \right) &= \left\{ \left( \young(12,3), \young(123) \right), \left( \young(123), \young(12,3) \right), \left( \young(1,2,3) , \young(13,2) \right), \left( \young(13,2) , \young(1,2,3) \right), \left( \young(13,2) , \young(13,2) \right) \right\}; \\
F_{3,2}^{-1}\left( \young(13,2) \right) &= \left\{ \left( \young(13,2), \young(123) \right), \left( \young(123), \young(13,2) \right), \left( \young(1,2,3) , \young(12,3) \right), \left( \young(12,3) , \young(1,2,3) \right), \left( \young(13,2) , \young(12,3) \right) \right\}; \\
F_{3,2}^{-1}\left( \young(1,2,3) \right) &= \left\{ \left( \young(123), \young(1,2,3) \right), \left( \young(12,3), \young(13,2) \right), \left( \young(1,2,3) , \young(123) \right) \right\}.
\end{align*}


For $m=4$, there is $10$ standard tableaux with $4$ cells, so there are $10^2 = 100$ pairs of tableaux to divide in $10$ sets. While it is not too difficult to construct a Kronecker map $F_{4,2}$ in that case, we can see that the difficulty increases rapidly.

Note that all possible pairs of tableaux will always appear exactly once.\\

Once $F_{m,2}$ has been found, then using proposition~\ref{Kronecker} recursively, we can define the map $F_{m,k}$ to be the map $F_{m,2} \circ F_{m,2} \times \text{Id} \circ \ldots \circ F_{m,2} \times \text{Id}^{m-1}$. Thus, the only difficulty to construct $F_{m,k}$ is to construct the map $F_{m,2}$.

\begin{ex}\label{ex:KroneckerMap}
Using the map $F_{3,2}$ as above, then:
\begin{align*}
F_{3,3}\left( \young(12,3),\young(13,2), \young(12,3) \right) &= F_{3,2}\left( F_{3,2} \left( \young(12,3),\young(13,2) \right), \young(12,3) \right) \\
&= F_{3,2} \left( \young(1,2,3), \young(12,3) \right) \\
&= \young(13,2).
\end{align*}

It means that the copy of $s_{21}[f_1]s_{21}[f_2]s_{21}[f_3]$ indexed respectively by tableaux $\young(12,3),\young(13,2)$ and $\young(12,3)$ contributes to the copy of $s_{21}[f_1f_2f_3]$ indexed by $\young(13,2)$.
\end{ex}


\subsection{The case $h_{\lambda}^m$}

Let $\lambda = (\lambda_1, \ldots, \lambda_k)$ be a partition. Recall that $h_{\lambda} = h_{\lambda_1} \ldots h_{\lambda_k}$ We denote \[\lambda^m = (\underbrace{\lambda_1, \ldots, \lambda_1}_{m \text{ times}}, \ldots, \underbrace{\lambda_k, \ldots, \lambda_k}_{m \text{ times}})\] the partition where each part is repeated $m$ times. We use this notation so that $h_{\lambda}^m = h_{\lambda^m}$. 

We show how to use the plethystic decomposition of $h_n^m$ and the map $F_{m,k}$ from the previous section to obtain the plethystic decomposition of $h_{\lambda}^m$. This has been done in \cite{MaasTetreault} for $m=2$, and we generalize this construction. This generalization is somewhat straightforward; see the article for a more detailed exposition.\\

By the Pieri rule, we have that:

\[
h_{\lambda}^m = \displaystyle \sum_{\mu \vdash m|\lambda|} K_{\lambda^m}^{\mu} s_{\mu}
\]

We can decompose a tableau of content $\lambda^m$ into a tuple of tableaux by considering subtableaux with some subsets of entries and using Schützenberger's \textit{jeu de taquin} \cite{Schutzenberger} 
to rectify them.
See \cite{Fulton} or \cite{Sagan} for more details about jeu de taquin and its properties. A jeu de taquin slide starts at an inner corner of the tableau, exchanges that empty cell with the cell directly under it or to its right, respecting the conditions on rows and columns for tableaux, and continues exchanges until the empty cell can no longer be moved. Below is illustrated the jeu de taquin slide starting at the inner corner in red:

\[
\gyoung(::;1,:!\rd;!\wt;2,:;1;3,;2;4)\rightarrow \gyoung(::;1,:;1;2,:!\rd;!\wt;3,;2;4)\rightarrow \gyoung(::;1,:;1;2,:;3!\rd;!\wt,;2;4).
\]

Doing this until there are no more inner corners, we obtain a (straight) tableau. Schützenberger proved \cite{Schutzenberger} that no matter the order of the slides applied of a skew tableau $T$, we always obtain the same \textit{rectification of $T$}, denoted $\text{Rect}(T)$. The rectification of the skew tableau above is $\young(11,22,3,4)$.

Using this, let $\mu \vdash m|\lambda|$ and $T \in \SSYT(\mu,\lambda^m)$. For $1 \leq i \leq k$, let $T^{(i)}$ be the skew tableau obtained by taking the cells with entries in $\{(i-1)m + 1, (i-1)m + 2, \ldots, im\}$. Then, let $\widetilde{T}^{(i)}$ be the tableau obtained by subtracting $(i-1)m$ to each entry; this way, $\widetilde{T}^{(i)}$ has content $(\lambda_i)^m$. Finally, let $T_i = \text{Rect}(\widetilde{T}^{(i)})$. We define the decomposition $\text{Dec}_m(T)$ to be the tuple $(T_1, \ldots, T_k)$. We can then determine the type of each $T_i$ according to theorem~1, and use the Kronecker map $F_{3,2}$ given in section~\ref{section:Kronecker} to get the type of $T$.

It is easy to derive from results in \cite{MaasTetreault} that:

\begin{prop}
Let $\lambda = (\lambda_1, \ldots, \lambda_k)$ be a partition and $m$ a positive integer. For $1 \leq i \leq k$, let $T_i$ be a tableau of content $(\lambda_i)^m$. Then, we have:
\[
\displaystyle \prod_{i=1}^k s_{\text{shape}(T_i)} = \sum_{T} s_{\text{shape}(T)},
\]
where the sum is over all $T$ of content $(\lambda)^m$ such that $\text{Dec}_m(T) = (T_1, \ldots, T_k)$.
\end{prop}

\begin{rem}
    In \cite{MaasTetreault}, we construct skew tableaux $T_1\ast T_2 \ast \dots \ast T_k$, where $T_i$ has filling $(\lambda_i)^2$, so only the numbner of different entries differs here. This skew tableau is then rectified using jeu de taquin to obtain the straight tableau $T$. The tableaux $T$ we sum on are then exactly tableaux of filling $\lambda^m$ which are rectifications of such skew tableaux.
\end{rem}

As a corollary, we obtain:

\begin{cor}
Suppose that we know the plethystic decomposition of $h_n^m$ for all $n$, \textrm{i.e.} we can attribute types to tableaux of content $(n)^m$. Then, for any $(t_1, \ldots, t_k) \in \SYT_m^k$, we have
\[
\prod_{i=1}^k s_{\text{shape}(t_i)}[h_{\lambda_i}] = \sum_{T} s_{\text{shape}(T)},
\]
where the sum is over all tableaux $T$ of content $\lambda^m$ such that if $\text{Dec}_m(T) = (T_1, \ldots, T_k)$, then $T_i$ has type $t_i$ for all $i$.
\end{cor}

Then, by using results from the previous section, we finally obtain that:

\begin{theo}
Suppose that we know the plethystic decomposition of $h_n^m$ for all $n$, and that we know a Kronecker map $F_{m,k}: \SYT_m^k \rightarrow \SYT_m$ (as defined in the previous section). Let $\lambda = (\lambda_1, \ldots, \lambda_k)$ be a partition and $T$ be a tableau of content $\lambda^m$. If $\text{Dec}_m(T) = (T_1, \ldots, T_k)$, let $t_i$ be the type of $T_i$. Define $t = F_{m,k}(t_1, \ldots, t_k)$ to be the type of $T$. Then:
\[
s_{\text{shape}(t)}[h_{\lambda}] = \displaystyle \sum_T s_{\text{shape}(T)},
\]
where the sum is over all tableaux $T$ of content $\lambda^m$ having type $t$.
\end{theo}

Thus, to understand the plethystic decomposition of $h_{\lambda}^m$, we only need to understand the plethystic decomposition of $h_n^m$ (for all $n$) and a Kronecker map $F_{2,m}$, and then apply the results of this section. This has been done in this article for $m=3$, so we can understand the plethystic decomposition of $h_{\lambda}^3$. See figure~\ref{fig:Dec} as an example, where the types are attributed as in section~\ref{section:Types}, and $F_{3,2}$ has been computed in example~\ref{ex:KroneckerMap}.

\begin{figure}[h]
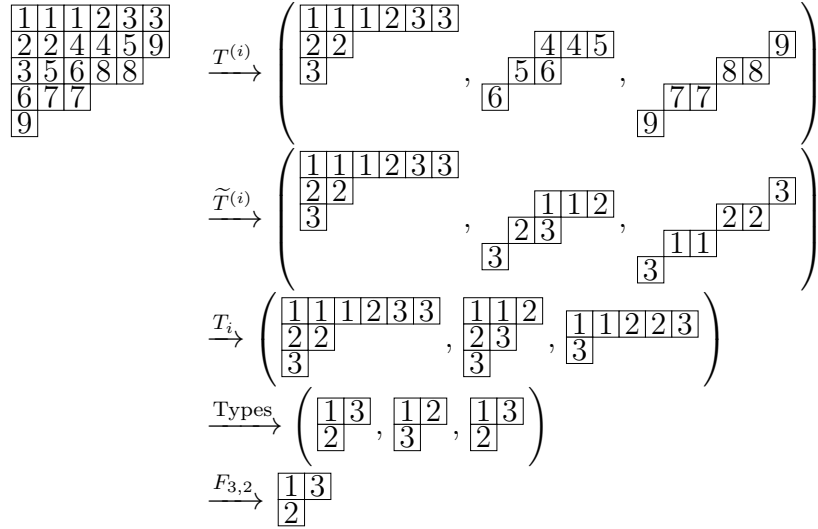

    \begin{center}
    \begin{tabular}{ll}
    \gyoung(;1;1;1;2;3;3,;2;2;4;4;5;9,;3;5;6;8;8,;6;7;7,;9)
    &$\xrightarrow{T^{(i)}} \left( \gyoung(;1;1;1;2;3;3,;2;2,;3,:,:), \gyoung(::::::,::;4;4;5,:;5;6,;6,:), \gyoung(::::::,:::::;9,:::;8;8,:;7;7,;9) \right)$ \\ 
    &$\xrightarrow{\widetilde{T}^{(i)}} \left( \gyoung(;1;1;1;2;3;3,;2;2,;3,:,:), \gyoung(::::::,::;1;1;2,:;2;3,;3), \gyoung(::::::,:::::;3,:::;2;2,:;1;1,;3) \right)$ \\ 
    &$\xrightarrow{T_i} \left( \gyoung(;1;1;1;2;3;3,;2;2,;3), \gyoung(;1;1;2,;2;3,;3), \gyoung(;1;1;2;2;3,;3) \right)$ \\
    &$\xrightarrow{\text{Types}} \left( \young(13,2), \young(12,3),\young(13,2) \right)$ \\
    &$\xrightarrow{F_{3,2}} \young(13,2)$\\
    \end{tabular}
    \end{center}
    \caption{Attributing a type to a tableau of shape $(6,6,5,3,1)$ and content $(3,2,2)^3$.}
    \label{fig:Dec}
\end{figure}

\subsection{The case $e_{\lambda}^m$}

The same technique can be applied to understand the plethystic decomposition of $e_{\lambda}^m$, by applying the $\omega$ involution. We have seen how to obtain the plethystic decomposition of $e_n^m$ from the one of $h_n^m$ in section~\ref{section:elementary}, by attributing types to conjugate tableaux of content $(n)^m$. 

So we have:

\[
e_{\lambda}^m = \omega(h_{\lambda}^m) = \displaystyle \sum_{\mu \vdash m|\lambda|} K_{\lambda^m}^{\mu^\prime} s_{\mu}.
\]

Recall that $K_{\lambda^m}^{\mu^{\prime}}$ counts the number of conjugate tableaux of shape $\mu$ and content $\lambda^{m}$. Let $T^{\prime}$ be such a conjugate tableau. As before, define ${T^{\prime}}^{(i)}$ to be the skew conjugate tableau obtained by taking the cells with entries in $\{(i-1)m+1, (i-1)m+2, \ldots, im\}$. Then, let $\widetilde{T^{\prime}}^{(i)}$ be the skew conjugate tableau obtained by subtracting $(i-1)m$ to every entry of ${T^{\prime}}^{(i)}$. Finally, let $T_{i}^{\prime} = {\text{Rect}({({\widetilde{T^{\prime}}}^{(i)})}{\prime} )}{\prime}$ (so we rectify the conjugate of $\widetilde{T^{\prime}}^{(i)}$, and conjugate again). We define $\text{Dec}_m^{\prime} (T^{\prime}) = (T_1^{\prime}, \ldots, T_k^{\prime})$. Then, we have the following theorem: 

\begin{theo}
Suppose that we know the plethystic decomposition of $e_n^{m}$ for all $n$ (see section~\ref{section:elementary}), and that we know a Kronecker map $F_{m,k}: \SYT_m^{k} \rightarrow \SYT_m$. Let $\lambda = (\lambda_1, \ldots, \lambda_k)$ be a partition, and $T^{\prime}$ be a conjugate tableau of content $\lambda^{m}$. If $\text{Dec}_m^{\prime} (T^{\prime}) = (T_1^{\prime}, \ldots, T_k^{\prime})$, let $t_i$ be the type of $T_i^{\prime}$. Define $t = F_{m,k}(t_1, \ldots, t_k)$ to be the type of $T^{\prime}$. Then,
\[
s_{\text{shape}(t)}[e_{\lambda}] = \displaystyle \sum_{T^{\prime}} s_{\text{shape}(T^{\prime})},
\]
where the sum is over all conjugate tableaux $T^{\prime}$ of content $\lambda^{m}$ having type $t$.
\end{theo}

This theorem is obtained by applying the $\omega$ involution on results of the previous section. 

\section{Conclusion}

We have shown in this article two possible plethystic decompositions for $h_n^3$, using various techniques. Moreover, we have shown how to extend these results to understand the plethystic decomposition of $e_n^3$, $h_{\lambda}^3$ and $e_{\lambda}^3$.\\

In the future, we hope to be able to understand the plethystic decomposition of $h_n^4$. Using results in this article, it would imply a plethystic decomposition of $e_n^4$, and by constructing a Kronecker map $F_{4,2}$, of $h_{\lambda}^4$ and $e_{\lambda}^4$. We also hope to be able to use these results for the plethystic decomposition of $s_{\lambda}^m$ when we know the one of $h_n^m$.

\section*{Acknowledgements}
The authors are grateful to Franco Saliola for his support, and to Mike Zabrocki who shared some of his related conjectures. FMG received support from NSERC.

\bibliographystyle{alpha}
\bibliography{main.bib}

\newpage

\renewcommand\thesection{\Alph{section}}
\setcounter{section}{1}

\section*{Appendix : Counting integer points on $\mathbb{P}_n$}
\label{annex:NumberOfIntegerPoints}

The integer points on the polytope $\mathbb{P}_n$, defined in section~\ref{section:polytopes} count tableaux with content $(n)^3$, and of any shape $\nu$. We have seen that $\mathbb{P}_n$ can be visualized through its projection onto the union of polytopes $P_m$, which represent tableaux of content $(n)^3$, and of shape $\nu$ with $\nu_3 = n-m$. Figure~\ref{fig:P_n} illustrates the first $P_n$'s, for $n=0,1,2,3,4,5,6$. Together, they give projections onto $\mathbb{P}_6$ at $\kappa_5 = 6,5,4,3,2,1,0$. \\

The numbers of points in these $P_m$'s are $1,3,7,12,19,27,37,$ etc., which give the numbers of points in $\mathbb{P}_n$ : $1,4,11,23,42,69,106,$ etc.
These two series are known on the OEIS as sequences $A077043$ and $A019298$. This interpretation of these sequences as counting tableaux however appears to be new. 

\begin{prop}
	Let $n\in \N$. The number of integer points on $\mathbb{P}_n$ is equal to the sum of integer point in all $P_m$ for $m\leq n$ :
	\[
	\#\{\text{ integer points } (\kappa_1,\kappa_2,\kappa_3,\kappa_4,\kappa_5) \text{ in } \mathbb{P}_n \ \} = 
	\sum_{0\leq m\leq n} \underbrace{\bigg[ \frac{(m+2)(m+1)}{2} + \Bigl\lfloor \frac{m}{2} \Bigr\rfloor \Bigl\lceil \frac{m}{2} \Bigr\rceil \bigg]}_{\normalsize \#\{\text{ integer points } (\kappa_1,\kappa_2,\kappa_3,\kappa_4,\kappa_5) \text{ in } P_m \ \}}
	\]
	\label{thm:numberofintegerpoints}
\end{prop}

The proof rests on the following lemmas. In particular, we show that the last part of the formula for $P_m$ gives the number of points which are added to incomplete stands when going from $P_{m-1}$ to $P_m$.

\begin{lem}
	Final points of strands in $P_n$ are integer points on lines $\kappa_1+\kappa_2 = n$ (so $\kappa_4=0)$ or on $\kappa_2+\kappa_4 = n$ with $\kappa_2\geq \kappa_4$ ( so $\kappa_1=0)$). In particular, $\kappa_3=0$ and they lie on $\kappa_1+\kappa_2+\kappa_4=n$.  \\
	
	The position of the final point of the $j$-strand in $P_n$ is $(\kappa_1,\kappa_2,\kappa_4) = (n-j,j,0)$ if $0\leq j\leq n$, and $(\kappa_1,\kappa_2,\kappa_4) = (0,2n-j,j-n)$ if $n\leq j\leq \Bigl\lfloor \frac{3n}{2} \Bigr\rfloor$.
\end{lem}

\begin{proof}
	Let's start by noting that all final points must lay on the lines $\kappa_1+\kappa_2 = n$ or $\kappa_2+\kappa_4 = n$.
	If $\kappa_3 = 1$, we can always apply the transformation $(\kappa_1,\kappa_2,\kappa_4)+(0,+1,0)$, which changes the boolean value of $\kappa_3$. So final points must have $\kappa_3=0$, and must lie on the polytope $\kappa_1+\kappa_2+\kappa_4 = n$.
	If neither $\kappa_1=0$ or $\kappa_4=0$ (and $\kappa_3=0$), then we can apply the two transformations $(\kappa_1,\kappa_2,\kappa_4)+(-1,+1,-1)+(0,+1,0)$ iteratively until one or both are zero. Then no more transformations can be applied, and the final point must lie on the wanted lines. Note that if any considered starting point has $\kappa_1+\kappa_2\geq \kappa_4$, then so will the final point.\\
	
	
	Recall that points $(\kappa_1,\kappa_3,\kappa_2,\kappa_4,\kappa_5)$ on $\mathbb{P}_n$ that give $\nu = (3n-j,j)$, are such that $\nu_3 = \kappa_5 = 0$, $\nu_2 = j = \kappa_2+\kappa_3+2\kappa_4+\kappa_5$ and $\nu_1 = 3n-j = 3\kappa_1+2\kappa_2+2\kappa_3+\kappa_4+\kappa_5$. Therefore, the final point of $(3n-j,j)$ in $P_n$ will be such that $j = \kappa_2+2\kappa_4$, and $3m-j = 3\kappa_1+2\kappa_2+\kappa_4$.\\
	
	Final point will have either $\kappa_1 = 0$, or $\kappa_4=0$, as seen above.\\
	
	If $\kappa_4 = 0$, then $j=\kappa_2$ and $\kappa_1 = n-j$, which is greater or equal to zero if and only if $0\leq j\leq n$.
	Therefore, the $m+1$ partitions $(3n-j,j)$ such that $0\leq j\leq n$ will have their final point on $\kappa_1+\kappa_2=n$ ($\kappa_4=0$), in position $(\kappa_1,\kappa_2,\kappa_4) = (m-j,j,0)$.
	
	Now, if $\kappa_1 = 0$, then $\kappa_2=2n-j$ and $\kappa_4 = j-n$, which is greater or equal to zero if and only if $n\leq j\leq \lfloor \frac{3n}{2} \rfloor$ .
	Therefore, the $\lfloor \frac{n}{2}\rfloor +1 $ partitions $(3n-j,j)$ such that $n\leq j\leq \lfloor \frac{3n}{2}\rfloor$ will have their final point on $\kappa_2+\kappa_4=n$ ($\kappa_1=0$), in position $(\kappa_1,\kappa_2,\kappa_4) = (0,2n-j,j-n)$.
\end{proof}

\begin{rem}
	Final points of $\nu=(3n-j,j)$ will lie on $\kappa_1+\kappa_2=n$ when $3n-2j = \nu_1-\nu_2\geq \nu_2-\nu_3 = j$, and on $\kappa_2+\kappa_4=n$ otherwise. 
\end{rem}


\begin{cor}
	The partitions $(3n-j,j)$ with $0\leq j\leq \lfloor \frac{3n}{2} \rfloor$ are in bijection with integer points in $P_n$ on the lines $\kappa_1+\kappa_2 = n$ ($\kappa_4=0$) and $\kappa_2+\kappa_4=n$ with $\kappa_2\geq \kappa_4$ ($\kappa_1 = 0$).
\end{cor}

\begin{lem}
	For $\nu = (3n-j,j)$ fixed, the value of $K^{(3n-j,j)}_{(n)^3} = \min(\nu_1-\nu_2,\nu_2-\nu_3)+1$ is given by $\kappa_2-\kappa_4+1$, for $(\kappa_1,\kappa_2,\kappa_4)$ the final point of the strand of $\nu$ in $P_n$.
\end{lem}

\begin{proof}
	We have seen that if $j\leq m$ then the final point is in position $(\kappa_1,\kappa_2,\kappa_4)=(m-j,j,0)$ in $P_n$. Then $K^\nu_{(n)^3} = \nu_2-\nu_3 +1 = j+1 = \kappa_2-\kappa_4+1$.\\
	
	Similarly, if $m\leq j \leq \lfloor \frac{3m}{2}\rfloor$, then the final point is in position $(\kappa_1,\kappa_2,\kappa_4)=(0,2n-j,j-n)$ in $P_n$. Then $K^\nu_{(n)^3} = \nu_1-\nu_2 +1 = 3n-2j +1 = \kappa_2 -\kappa_4+1$.
\end{proof}

\begin{proof}[Proof of Proposition~\ref{thm:numberofintegerpoints}]
	
	The $n+1$ strands in $P_n$ associated with partitions $(3n-j,j)$ with $0\leq j\leq n$ have their final points on line $\kappa_1+\kappa_2=n$, and have $K^{(3n-j,j)}_{(n)^3} = j+1$. Therefore the number of integer points on $P_n$ in these strands is $\displaystyle \sum_{0\leq j\leq n} j+1 = \frac{(n+1)(n+2)}{2}$.\\
	
	The $\lceil \frac{n}{2}\rceil$ strands associated with partitions $(3n-j,j)$ with $n< j\leq \lfloor \frac{3n}{2}\rfloor$ have their final point on the line $\kappa_2+\kappa_4=n$, with $\kappa_2\geq \kappa_4$. Then they have $K^{(3n-j,j)}_{(n)^3}=3n-2j+1 = n+1-2\ell$, where $\ell = j-n$.
	
	The number of point on these strands is $\displaystyle \sum_{1\leq \ell \leq \lfloor \frac{n}{2}\rfloor} n+1-2\ell  = \lfloor \frac{n}{2}\rfloor(n-\lfloor \frac{n}{2}\rfloor) = \lfloor \frac{n}{2}\rfloor\lceil \frac{n}{2}\rceil$.
	This gives the wanted result for the number of points in $P_n$.
	
	As discussed above, $\mathbb{P}_n$ is constructed from all the $P_m$ with $m\leq n$, by letting $\kappa_5 = n-m$. Therefore, we have the wanted result for $\mathbb{P}_n$ as well.
\end{proof}

\end{document}